\RequirePackage{fix-cm}
\documentclass[envcountsect,smallextended]{svjour3}       
\smartqed
\usepackage{graphicx}
\usepackage{arevmath}
\usepackage[utf8]{inputenc}
\usepackage[T1]{fontenc}
\usepackage{lmodern}
\usepackage{amsmath}
\usepackage{amssymb}
\usepackage{marginnote}
\usepackage{xcolor}
\usepackage{fancyhdr}
\usepackage{esint}
\usepackage{tikz} 
\usepackage[arrow, matrix, curve]{xy}
\usepackage{hyperref}
\usetikzlibrary{fit,positioning} 
\usepackage{geometry}
\usepackage{pifont}
\usepackage{nccmath}
\usepackage{tocloft}
\usepackage{setspace}
\usepackage{graphicx}
\usepackage{calligra}
\usepackage[english]{babel}

\usepackage{blindtext}

\newcommand{\vertiii}[1]{{\vert\kern-0.25ex\vert\kern-0.25ex\vert #1 
		\vert\kern-0.25ex\vert\kern-0.25ex\vert}}
	
	\usepackage{color}
	\definecolor{Darkgreen}{rgb}{0,0.4,0}
	\usepackage{hyperref}
	\hypersetup{%
		pdfborder = {0 0 0},
		colorlinks,
		citecolor=blue,
		filecolor=black,
		linkcolor=blue,
		urlcolor=cyan!50!black!90
	}

\DeclareMathAlphabet\mathbfcal{OMS}{cmsy}{b}{n}
\DeclareMathAlphabet\mathbfscr{OMS}{mdugm}{b}{n}

\spnewtheorem{defn}[equation]{Definition}{\bfseries}{\upshape}
\spnewtheorem{prop}[equation]{Proposition}{\bfseries}{\upshape}
\spnewtheorem{thm}[equation]{Theorem}{\bfseries}{\upshape}
\spnewtheorem{cor}[equation]{Corollary}{\bfseries}{\upshape}
\spnewtheorem{rmk}[equation]{Remark}{\bfseries}{\upshape}
\spnewtheorem{lem}[equation]{Lemma}{\bfseries}{\upshape}
\spnewtheorem{expl}[equation]{Example}{\bfseries}{\upshape}

\numberwithin{equation}{section} 

\begin{document}
	
	\title{Note on the existence theory for non-induced evolution equations
	}
	
	\author{A. Kaltenbach}

	\institute{A. Kaltenbach \at
		Institute of Applied Mathematics, Albert-Ludwigs-University Freiburg, Ernst-Zermelo-Straße 1, 79104 Freiburg,\\
		\email{alex.kaltenbach@mathematik.uni-freiburg.de}           
	}
	
	\date{Received: date / Accepted: date}

	\maketitle
	
	\begin{abstract}
		In this note we develop a framework which allows to 
		prove an abstract existence result
		for non-linear evolution equations involving so-called 
		non-induced operators, i.e., operators which are not prescribed by a time-dependent family of operators.
		Apart from this, we introduce the notion of $C^0$-Bochner pseudo-monotonicity,
		and $C^0$-Bochner coercivity, which are appropriate adaptions of the standard
		notion to the case of evolutionary problems. 
		\keywords{Evolution equation \and Pseudo-monotone operator \and Existence result}
	\end{abstract}

\section{Introduction}
\label{intro}
The theory of pseudo-monotone operators proved itself as a reliable tool 
in the verification of the solvability of non-linear problems. At its core lies 
the main theorem on pseudo-monotone operators, tracing back to Brezis 
\cite{Bre68}, which states the following\footnote{All
	notion are defined in Section~\ref{sec:2}.} 

\begin{thm}\label{1.1}
	Let $(X,\|\cdot\|_X)$ be a reflexive Banach space and $A:X\rightarrow X^*$ a bounded, pseudo-monotone and coercive operator. Then $R(A)=X^*$.
\end{thm} 

A remarkable number of contributions, see e.g.,
 \cite{LS65,Lio69,GGZ74,Zei90B,Pap97,Shi97,Sho97,Ru04,Rou05,BR17,KR19}, 
 dealt with the question to what extent Brezis' result is transferable 
 to the framework of non-linear evolution equations.
  A popular time-dependent analogue of Brezis' contribution is the following (cf.~\cite{Lio69,Zei90B,Ru04,Rou05})
  
  \begin{thm}\label{1.2}
  	Let $(V,H,j)$ be an evolution triple, $I:=\left(0,T\right)$ a finite time horizon, ${\boldsymbol{y}_0\in H}$ an initial value, $\boldsymbol{f}\in L^{p'}(I,V^*)$, $1<p<\infty$, a right-hand side and $\mathbfcal{A}:L^p(I,V)\rightarrow L^{p'}(I,V^*)$ a bounded, pseudo-monotone and coercive operator. Then there exists a solution 
  	$\boldsymbol{y}\in W^{1,p,p'}(I,V,V^*)$ of the initial value problem
  	\begin{align}
  	\begin{split}
  	\begin{alignedat}{2}
  	\frac{d\boldsymbol{y}}{dt}+\mathbfcal{A}\boldsymbol{y}&
  	=\boldsymbol{f}&&\quad\text{ in }L^{p'}(I,V^*),\\j(\boldsymbol{y}(0))&
  	=\boldsymbol{y}_0&&\quad\text{ in }H.
  	\end{alignedat}
  	\end{split}\label{eq:1.1}
  	\end{align}
  \end{thm}

A major drawback of Theorem \ref{1.2} is that the number of non-monotone, but pseudo-monotone, operators ${\mathbfcal{A}:L^p(I,V)\rightarrow L^{p'}(I,V^*)}$ is negligible and thus the scope of application of Theorem~\ref{1.2} is strictly limited. For example, consider the unsteady $p$-Navier-Stokes equations, which can be written as an initial value problem of type \eqref{eq:1.1}, where $V=W^{1,p}_{0,\text{div}}(\Omega)$\footnote{$W^{1,p}_{0,\text{div}}(\Omega)$ is the closure of $\mathcal{V}=\{v\in C^\infty(\Omega)^3\mid \text{div}v\equiv 0\}$ with respect to $\|\nabla\cdot\|_{L^p(\Omega)^{3\times 3}}$ and $L^2_{\text{div}}(\Omega)$ the closure of $\mathcal{V}$ with respect to $\|\cdot\|_{L^2(\Omega)^3}$.}, ${p>3}$, $H=L^2_{\text{div}}(\Omega)$ and $\mathbfcal{A}=\mathbfcal{S}+\mathbfcal{B}: L^p(I,V)\to L^{p'}(I,V^*)$, given via $\langle\mathbfcal{S}\boldsymbol{x},\boldsymbol{y}\rangle_{L^p(I,V)}\!=\!\int_I\int_{\Omega}{\!(\delta+\vert\boldsymbol{D}\boldsymbol{x}\vert^{p-2})\boldsymbol{D}\boldsymbol{x}:\boldsymbol{D}\boldsymbol{y}\,dxdt}$,\footnote{$\textbf{D}\boldsymbol{y}=\frac{1}{2}(\nabla\boldsymbol{y}+(\nabla\boldsymbol{y})^\top )$ denotes the symmetric gradient.} $\delta\ge 0$, and $\langle\mathbfcal{B}\boldsymbol{x},\boldsymbol{y}\rangle_{L^p(I,V)}=-\int_I\int_{\Omega}{\boldsymbol{x}\otimes\boldsymbol{x}:\boldsymbol{D}\boldsymbol{y}\,dxdt}$
for all $\boldsymbol{x},\boldsymbol{y}\in L^p(I,V)$. While $\mathbfcal{S}: L^p(I,V)\to L^{p'}(I,V^*)$ is monotone, continuous, and thus pseudo-monotone, $\mathbfcal{B}: L^p(I,V)\to L^{p'}(I,V^*)$ fails to be pseudo-monotone (cf. \cite[Remark 3.7]{KR19}). Therefore, Theorem \ref{1.2} is not applicable on the unsteady $p$-Navier-Stokes equations.

In \cite{Lio69} J.-L. Lions already observed that incorporating information from the time derivative will help to overcome this restriction. To this end, he introduced the following generalization of pseudo-monotonicity.

\begin{defn}[$\frac{d}{dt}$-pseudo-monotonicity]\label{1.21}
	Let $(V,H,j)$ be an evolution triple, $I:=\left(0,T\right)$, with $0<T<\infty$, and $1<p<\infty$. An operator $\mathbfcal{A}:W^{1,p,p'}(I,V,V^*)\to L^{p'}(I,V^*)$ is said to be \textbf{$\frac{d}{dt}$-pseudo-monotone} if for a sequence $(\boldsymbol{x}_n)_{n\in\mathbb{N}}\subseteq W^{1,p,p'}(I,V,V^*)$ from
	\begin{align}
	\boldsymbol{x}_n\overset{n\rightarrow\infty}{\rightharpoonup}
	\boldsymbol{x}\quad\text{ in }W^{1,p,p'}(I,V,V^*),\label{eq:1.9}
	\\
	\limsup_{n\rightarrow\infty}\langle\mathbfcal{A}\boldsymbol{x}_n,
	\boldsymbol{x}_n-\boldsymbol{x}\rangle_{L^p(I,V)}\leq
	0,\label{eq:1.10}
	\end{align}
	it follows that $ \langle\mathbfcal{A}\boldsymbol{x},\boldsymbol{x}
	-\boldsymbol{y}\rangle_{L^p(I,V)}
	\leq\liminf_{n\rightarrow\infty}{\langle\mathbfcal{A}\boldsymbol{x}_n,\boldsymbol{x}_n
		-\boldsymbol{y}\rangle_{L^p(I,V)}}\text{ for all }\boldsymbol{y}\in L^p(I,V)$. 
\end{defn}

With this notion Lions was able to extend Theorem \ref{1.2} to $\frac{d}{dt}$-pseudo-monotone, coercive operators $\mathbfcal{A}:W^{1,p,p'}(I,V,V^*)\to L^{p'}(I,V^*)$ satisfying a special boundedness condition which takes the time derivative into account (cf.~\cite[Th\'eor\'eme 1.2, p. 316]{Lio69}). In fact, he proved that $\mathbfcal{S}+\mathbfcal{B}: L^p(I,V)\to L^{p'}(I,V^*)$ is $\frac{d}{dt}$-pseudo-monotone, coercive and satisfies this special boundedness condition (cf.~\cite[Remarque 1.2, p. 335]{Lio69}).
Unfortunately, \cite[Th\'eor\'eme 1.2]{Lio69} is entailing an imbalance between the demanded continuity and growth conditions. To be more precise, while the required $\frac{d}{dt}$-pseudo-monotonicity is quite general, coercivity is a restrictive assumption, which in many application is not fulfilled, e.g., $\mathbfcal{S}-\mathbfcal{R}:L^p(I,V)\rightarrow L^{p'}(I,V^*)$, where $\mathbfcal{S}:L^p(I,V)\rightarrow L^{p'}(I,V^*)$ is defined as above and  $\mathbfcal{R}:L^p(I,V)\to L^{p'}(I,V^*)$ is given via $\langle \mathbfcal{R}\boldsymbol{x},\boldsymbol{y}\rangle_{L^p(I,V)}:=\int_I\int_{\Omega}{\boldsymbol{x}\cdot\boldsymbol{y}\,dxdt}$ for every $\boldsymbol{x},\boldsymbol{y}\in L^p(I,V)$, is $\frac{d}{dt}$-pseudo-monotone, but not coercive.

In \cite{KR19} this restriction is overcome by introducing alternative generalizations of pseudo-monotonicity and coercivity, which, in contrast to $\frac{d}{dt}$-pseudo-monotonicity and coercivity, both incorporate information from the time derivative, and therefore are more in balance. The idea is to weaken the pseudo-monotonicity assumption to a bearable extend, in order to make a coercivity condition accessible, which takes the information from the time derivative into account. 

\begin{defn}[Bochner pseudo-monotonicity and Bochner coercivity]\label{1.3} 
	Let $(V,H,j)$ be an evolution triple, $I:=\left(0,T\right)$, with $0<T<\infty$, and $1<p<\infty$.
		An operator $\mathbfcal{A}:L^p(I,V)\cap L^\infty(I,H)\to L^{p'}(I,V^*)$ is said to be
		\begin{description}
			\item[(i)] \textbf{Bochner pseudo-monotone} if for a sequence $(\boldsymbol{x}_n)_{n\in\mathbb{N}}\subseteq L^p(I,V)\cap L^\infty(I,H)$ from
			\begin{align}
			\boldsymbol{x}_n&\overset{n\rightarrow\infty}{\rightharpoonup}
			\boldsymbol{x}\;\;\;\;\;\;\;\;\,\quad\text{ in  }L^p(I,V),\label{eq:1.2}
			\\
			j(\boldsymbol{x}_n)&\;\;\overset{\ast}{\rightharpoondown}\;\;
			j(\boldsymbol{x})\;\;\;\;\quad\text{ in }L^\infty(I,H)\quad(n\rightarrow\infty),\label{eq:1.3}
			\\
			j(\boldsymbol{x}_n(t))&\overset{n\rightarrow\infty}{\rightharpoonup}
			j(\boldsymbol{x}(t))\quad\text{
				in }H\quad\text{for a.e. }t\in I,\label{eq:1.4}\\
			\limsup_{n\rightarrow\infty}\langle\mathbfcal{A}\boldsymbol{x}_n&,
			\boldsymbol{x}_n-\boldsymbol{x}\rangle_{L^p(I,V)}\leq
			0,\label{eq:1.5}
			\end{align}
			it follows that $ \langle\mathbfcal{A}\boldsymbol{x},\boldsymbol{x}
			-\boldsymbol{y}\rangle_{L^p(I,V)}
			\leq\liminf_{n\rightarrow\infty}{\langle\mathbfcal{A}\boldsymbol{x}_n,\boldsymbol{x}_n
				-\boldsymbol{y}\rangle_{L^p(I,V)}}\text{ for all }\boldsymbol{y}\in L^p(I,V)$. 
			\item[(ii)] \textbf{Bochner coercive with respect to $\boldsymbol{f}\in L^{p'}(I,V^*)$ and $\boldsymbol{x}_0\in H$}, if there exists a constant $M:=M(\boldsymbol{f},\boldsymbol{x}_0,\mathbfcal{A})>0$ 
			such that for all $\boldsymbol{x}\in L^p(I,V)\cap L^\infty(I,H)$ 
			from
			\begin{align*}
			\frac{1}{2}\|(\boldsymbol{j}\boldsymbol{x})(t)\|_H^2+\langle\mathbfcal{A}\boldsymbol{x}-
			\boldsymbol{f},\boldsymbol{x}\chi_{\left[0,t\right]}\rangle_{L^p(I,V)}\leq \frac{1}{2}\|\boldsymbol{x}_0\|_H^2\quad
			\text{ for a.e. }t\in I
			\end{align*}
			it follows that $\|\boldsymbol{x}\|_{L^p(I,V)\cap L^\infty(I,H)}\leq M$.
		\end{description}
\end{defn}

Bochner pseudo-monotonicity and Bochner coercivity in \cite{KR19} turned out to be  appropriate generalizations of pseudo-monotonicity and coercivity
for evolution equations since they both take into account the additional 
information from the time derivative, coming from the generalized 
integration by parts formula (cf. Proposition \ref{2.12}). In fact, in \cite{KR19} it is illustrated that \eqref{eq:1.2}--\eqref{eq:1.5} are natural properties of a sequence ${(\boldsymbol{x}_n)_{n\in\mathbb{N}}\subseteq L^p(I,V)\cap L^\infty(I,H)}$ coming from an appropriate Galerkin approximation of \eqref{eq:1.1}. To be more precise,
\eqref{eq:1.2} and \eqref{eq:1.3} result from the Bochner coercivity of $\mathbfcal{A}$, and therefore take into account information from both the operator and the time derivative (cf.~discussion below Definition \ref{3.7}), while \eqref{eq:1.4}
and \eqref{eq:1.5} follow directly from the Galerkin
approximation. In this way, \cite[Theorem 4.1]{KR19} provides an existence result for the initial value problem \eqref{eq:1.1} provided that ${\mathbfcal{A}:L^p(I,V)\cap L^\infty(I,H) \rightarrow L^{p'}(I,V^*)}$
  is Bochner pseudo-monotone, Bochner coercive and \textbf{induced} by a time-dependent family of
  operators ${A(t):V\rightarrow V^*}$, $t\in I$, i.e., 
  $(\mathbfcal{A}\boldsymbol{x})(t):=A(t)(\boldsymbol{x}(t))$ in $V^*$ for almost every 
  $t\in I$ and all $\boldsymbol{x}\in L^p(I,V)\cap L^\infty(I,H)$, satisfying 
  appropriate growth conditions 
  (cf.~\cite[Conditions (C.1)--(C.3)]{KR19}). Note that both $\mathbfcal{S}+\mathbfcal{B}: L^p(I,V)\to L^{p'}(I,V^*)$ and $\mathbfcal{S}-\mathbfcal{R}: L^p(I,V)\to L^{p'}(I,V^*)$ are Bochner pseudo-monotone and Bochner coercive (cf. \cite[Example 5.1]{KR19}), and  \cite[Theorem 4.1]{KR19} applicable.
  
  However, there are still non-negligible disadvantages of \cite[Theorem 4.1]{KR19} in comparison to Theorem \ref{1.2} and \cite[Th\'eor\'eme 1.2]{Lio69}, which consist in its non-applicability on non-induced
  operators and the needed separability of $V$. The necessity of an induced operator can be traced back to the verification of the existence of Galerkin solutions which in \cite{KR19} is based
  on Carath\'eodory's existence theorem for ordinary differential 
  equations (cf.~\cite[Theorem 5.1]{Hal80}) and the usual associated
  extension argument providing global in time existence. The latter argument 
  requires that $\mathbfcal{A}$ is a \textbf{Volterra operator}, i.e., 
  $\boldsymbol{x}=\boldsymbol{y}$ on $\left[0,t\right)$ for all 
  $t\in \overline{I}$ implies $\mathbfcal{A}\boldsymbol{x}=\mathbfcal{A}
  \boldsymbol{y}$ on $\left[0,t\right)$ for all $t\in \overline{I}$ 
  (cf. \cite[Kap. V, Definition 1.1]{GGZ74}). The separability of $V$ yields the existence of an increasing sequence of finite dimensional subspaces which approximates $V$ up to density. This increasing structure seems to be indispensable for the extraction of \eqref{eq:1.4} from the Galerkin approach applied in \cite{KR19}. 
  
  The main purpose of this paper is to remove these limitations and extend the new gap-filling concepts of \cite{KR19} to the abstract level of Theorem \ref{1.2} and \cite[Th\'eor\'eme 1.2]{Lio69}, i.e., to prove an existence result for non-induced, bounded, Bochner pseudo-monotone and Bochner coercive operators in the case of purely reflexive $V$, in order to gain a proper alternative to \cite[Th\'eor\'eme 1.2]{Lio69} and a generalization of \cite[Theorem 4.1]{KR19} and Theorem \ref{1.2}. To this end, we will combine the 
  modi operandi of \cite{Bro68}, \cite{BH72} and \cite{KR19}. To be more specific, 
  as one fails to extract \eqref{eq:1.4} from the Galerkin approximation method 
  applied in \cite[Theorem 1]{Bro68}, we are forced to fall back on the usual 
  Galerkin approach as in \cite{KR19}. 
  Therefore, we initially limit ourselves to the case of separable, reflexive $V$ and 
  extend this result to the case of purely reflexive $V$ by techniques from 
  \cite{Bro68,BH72} afterwards.
  
  A further intention of this paper is to point out that there is still space for generalizations of Bochner pseudo-monotonicity. Indeed, since \eqref{eq:1.3} together with \eqref{eq:1.4} is strictly 
  weaker than weak convergence in $L^\infty(I,H)$ (cf.~\cite[Remark 3.5]{Tol18} or Remark \ref{rmk}), one may be reluctant to require the latter in Definition \ref{1.3}. However, in
  \cite{KR19} it is shown that a sequence of Galerkin approximations
  in $L^p(I,V)\cap C^0(\overline{I},H) $ satisfies \eqref{eq:1.3} and \eqref{eq:1.4}
  not only for almost every, but for all $t\in\overline{I}$, which is
  equivalent to weak convergence in $C^0(\overline{I},H)$ 
  (cf.~\cite[Theorem 4.3]{BT38}). This suggests a generalization of Bochner 
  pseudo-monotonicity that respects the weak sequential topology in $C^0(\overline{I},H)$. 
  Therefore, we say that an operator
  $\mathbfcal{A}:L^p(I,V)\cap C^0(\overline{I},H)\rightarrow L^{p'}(I,V^*)$ is
  \textbf {$C^0$-Bochner pseudo-monotone} if from \eqref{eq:1.5} and
  \begin{align}
  \boldsymbol{x}_n\overset{n\rightarrow\infty}{\rightharpoonup}
  \boldsymbol{x}\text{ in  }L^p(I,V)\cap C^0(\overline{I},H),\label{eq:1.6}
  \end{align}
  it follows that
  ${\langle\mathbfcal{A}\boldsymbol{x},\boldsymbol{x} -
  	\boldsymbol{y}\rangle_{L^p(I,V)} \leq
  	\liminf_{n\rightarrow\infty}{\langle\mathbfcal{A}\boldsymbol{x}_n,
  		\boldsymbol{x}_n-\boldsymbol{y}\rangle_{L^p(I,V)}}\text{ for all }\boldsymbol{y}\in L^p(I,V)}$. 
  We will see that Bochner pseudo-monotonicity implies $C^0$-Bochner pseudo-monotonicity, but the converse is not true in general (cf.~Remark \ref{rmk}). In the same spirit, 
  we introduce $C^0$-Bochner condition (M) and $C^0$-Bochner coercivity as appropriate generalizations
  of the condition (M) and coercivity for evolution equations, as they take the additional energy 
  space $C^0(\overline{I},H)$ into account. 
  
Altogether, we prove an existence result for bounded, $C^0$-Bochner pseudo-monotone and $C^0$-Bochner coercive operators $\mathbfcal{A}:L^p(I,V)\cap C^0(\overline{I},H)\rightarrow L^{p'}(I,V^*)$, including also non-induced operators, even in the case of purely reflexive $V$.
 Note that any bounded and coercive, or Bochner coercive, operator
  $\mathbfcal{A}:L^p(I,V)\rightarrow L^{p'}(I,V^*)$ is $C^0$-Bochner coercive (cf.~Proposition \ref{3.8}), 
  and that Bochner pseudo-monotonicity or usual pseudo-monotonicity imply
  $C^0$-Bochner pseudo-monotonicity (cf.~Remark \ref{rmk}). We will thus gain a proper 
  generalization of both Theorem~\ref{1.2} and \cite{KR19}.

\textbf{Plan of the paper:} In Section \ref{sec:2} we introduce
the notation and some basic definitions and results concerning continuous 
functions, Bochner-Lebesgue spaces, Bochner-Sobolev spaces and evolution
equations. In Section \ref{sec:3} we introduce the new notions $C^0$-Bochner
pseudo-monotonicity, $C^0$-Bochner condition (M) and $C^0$-Bochner 
coercivity. In Section \ref{sec:4} we specify the implemented Galerkin approach.
In Section \ref{sec:5} we prove in the case of separable and reflexive $V$ an existence result for
evolution equations with not necessarily induced, bounded and $C^0$-Bochner coercive operators 
satisfying the $C^0$-Bochner condition (M).
Section \ref{sec:6} extends the results of Section \ref{sec:5} to the case of purely reflexive $V$.

The paper is an extended and modified version of parts of the thesis \cite{alex-master}.

\section{Preliminaries}
\label{sec:2}

\subsection{Operators}
For a Banach space $X$ with norm $\|\cdot\|_X$ we denote by $X^*$ its 
dual space equipped with the norm $\|\cdot\|_{X^*}$, by $\tau(X,X^*)$ 
the corresponding weak topology and by $B_M^X(x)$ the closed ball 
with radius $M>0$ and centre $x\in X$. The duality pairing is denoted 
by $\langle\cdot,\cdot\rangle_X$. All occurring Banach spaces are assumed 
to be real. By $D(A)$ we denote the domain of definition of an operator 
$A:D(A)\subseteq X\rightarrow Y$, and by $R(A):=\{Ax\mid x\in D(A)\}$ its range.
\begin{defn}\label{2.1}
	Let $(X,\|\cdot\|_X)$ and $(Y,\|\cdot\|_Y)$ be Banach spaces. 
	The operator $A:D(A)\subseteq X\rightarrow Y$ is said to be
		\begin{description}[{(viii)}]
			\item[{(i)}] \textbf{demi-continuous}, if $D(A)=X$, and
			$x_n\overset{n\rightarrow\infty}{\rightarrow}x$ in $X$ 
			implies $Ax_n\overset{n\rightarrow\infty}{\rightharpoonup}Ax$ in $Y$.
			\item[{(ii)}] \textbf{strongly continuous}, if $D(A)=X$, and
			$x_n\overset{n\rightarrow\infty}{\rightharpoonup}x$ in $X$
			 implies $Ax_n\overset{n\rightarrow\infty}{\rightarrow}Ax$ in $Y$.
			\item[{(iii)}] \textbf{compact}, if 
			$A:D(A)\subseteq X\rightarrow Y$ is continuous and 
			for all bounded $M\subseteq D(A)\subseteq X$ 
			the image $A(M)\subseteq Y$ is relatively compact.
			\item[{(iv)}] \textbf{bounded}, if for all bounded 
			$M\subseteq D(A)\subseteq X$ the image $A(M)\subseteq Y$ is bounded.
		\item[{(v)}] \textbf{locally bounded}, if for all $x_0\in D(A)$ there 
		exist constants $\varepsilon(x_0), \delta(x_0)>0$ such that 
		$\|Ax\|_Y\leq \varepsilon(x_0)$ for all $x\in D(A)$ with $\|x-x_0\|_X\leq \delta(x_0)$.
			\item[{(vi)}] \textbf{monotone}, if $Y=X^*$ 
			and $\langle Ax-Ay,x-y\rangle_X\ge 0$ for all $x,y\in D(A)$.
			\item[{(vii)}] \textbf{pseudo-monotone}, if $Y=X^*$, $D(A)=X$ and for $(x_n)_{n\in\mathbb{N}}\subseteq X$ from $x_n\overset{n\to\infty}{\rightharpoonup}x$ in $X$ and $\limsup_{n\to\infty}{\langle Ax_n,x_n-x\rangle_X}\leq 0$, it follows $\langle Ax,x-y\rangle_X\leq \liminf_{n\to\infty}{\langle Ax_n,x_n-y\rangle_X}$ for every $y\in X$.
			\item[{(viii)}] \textbf{coercive}, if $Y=X^*$, 
			$D(A)$ is unbounded and $	\lim_{\substack{\|x\|_X\rightarrow\infty\\x\in D(A)}}
			{\frac{\langle Ax,x\rangle_X}{\|x\|_X}}=\infty$.
		\end{description}
\end{defn}
The following proposition states that monotone operators satisfy 
certain boundedness properties and motivates to consider 
non-bounded operators provided that they are monotone.
\begin{prop}\label{2.2}
	Let $(X,\|\cdot\|_X)$ be a Banach space and 
	$A:X\rightarrow X^*$ monotone. Then it holds:
	\begin{description}[{(ii)}]
		\item[{(i)}] $A:X\rightarrow X^*$ is locally bounded.
		\item[{(ii)}] Let $S\subseteq X$ be a set, $h:S\rightarrow\left[0,1\right]$ 
		a function and $M,C>0$ constants such that for all $s\in S$
		\begin{align*}
		\|s\|_X\leq M\qquad\text{ and }\qquad h(s)\langle As,s\rangle_X\leq C.
		\end{align*}
		Then there exists a constant $K:=K(C,M,A)>0$ such that $\|h(s)As\|_{X^*}\leq K$
		for all $s\in S$.
	\end{description}
\end{prop}

\begin{proof}
	Concerning point (i) we refer to \cite[Kapitel III, Lemma 1.2]{GGZ74}. 
	Point (ii) is a modification of \cite[Kapitel III, Folgerung 1.2]{GGZ74}. 
	Being more precise, the locally boundedness of 
	$A:X\rightarrow X^*$ provides constants $\varepsilon, \delta>0$ 
	such that $\|Ax\|_{X^*}\leq \varepsilon$ for all $x\in X$
	with $\|x\|_X\leq \delta$. With the help of a scaled version 
	the norm formula we finally obtain for all $s\in S$
	\begin{align*}
	\|h(s)As\|_{X^*}=\sup_{\|x\|_X= \delta}{h(s)\frac{\langle As,x\rangle_X}{\delta}}
	\leq\sup_{\|x\|_X= \delta}{h(s)\frac{\langle As,s\rangle_X+\langle Ax,x-s\rangle_X}{\delta}}
	\leq\frac{C+\varepsilon\left(\delta+M\right)}{\delta},
	\end{align*}
	where we exploited the monotonicity of $A:X\rightarrow X^*$ in the first inequality.\hfill$\square$
\end{proof}

\subsection{Continuous functions and Bochner-Lebesgue spaces}
In this passage we collect some well-known results concerning 
continuous functions and Bochner-Lebesgue spaces, which will 
find use in the following. By $(X,\|\cdot\|_X)$ and $(Y,\|\cdot\|_Y)$ 
we always denote Banach spaces and by $I:=\left(0,T\right)$, 
with $0<T<\infty$, a finite time interval.
The first proposition serves in parts as motivation for 
$C^0$-Bochner pseudo-monotonicity.

\begin{prop}\label{2.3}
	It holds $\boldsymbol{x}_n\overset{n\rightarrow\infty}{\rightharpoonup}
	\boldsymbol{x}$ in $C^0(\overline{I},X)$ if and only 
	if $(\boldsymbol{x}_n)_{n\in\mathbb{N}}\subseteq C^0(\overline{I},X)$ 
	is bounded and $\boldsymbol{x}_n(t)\overset{n\rightarrow\infty}{\rightharpoonup}
	\boldsymbol{x}(t)\text{ in }X$
	for all $t\in\overline{I}$.
\end{prop}

\begin{proof}
	See \cite[Theorem 4.3]{BT38}.\hfill$\square$
\end{proof}

\begin{prop}\label{2.4}
	Let $1\leq p\leq \infty$ and let $A:X\rightarrow Y$ be linear and continuous. 
	Then the induced operator $\mathbfcal{A}:L^p(I,X)\rightarrow L^p(I,Y)$,
	 defined by $(\mathbfcal{A}\boldsymbol{x})(t):=A(\boldsymbol{x}(t))$ in $Y$
	 for almost every $t\in I$ and all $\boldsymbol{x}\in L^p(I,X)$, 
	 is well-defined, linear and continuous. Furthermore, it holds:
	\begin{description}[{(iii)}]
		\item[{(i)}] $A\left(\int_I{\boldsymbol{x}(s) \;ds}\right)=
		\int_I{(\mathbfcal{A}\boldsymbol{x})(s)\;ds}\text{ in }Y$ 
		for all $\boldsymbol{x}\in L^p(I,X)$.
		\item[{(ii)}] If $A:X\rightarrow Y$ is an embedding, 
		then also $\mathbfcal{A}:L^p(I,X)\rightarrow L^p(I,Y)$ is an embedding. 
		\item[{(iii)}] If $A:X\rightarrow Y$ is an isomorphism, 
		then also $\mathbfcal{A}:L^p(I,X)\rightarrow L^p(I,Y)$ is an isomorphism.
	\end{description}
\end{prop}

\begin{proof}
	 Concerning the well-definedness, linearity and boundedness 
	 including point (i) we refer to \cite[Chapter V, 5. Bochner's Integral, Corollary 2]{Yos80}. 
	 The verification of assertions (ii) and (iii) is elementary and thus omitted.\hfill$\square$ 
\end{proof}

We use the in \cite{KR19} proposed alternative 
point of view concerning intersections of Banach spaces, 
which is specified in the Appendix. We emphasize that the 
standard definition of intersections of Banach spaces 
(cf.~\cite{BS88}) is equivalent to our approach and all 
the following assertions remain true if we use the framework
 in \cite{BS88}. The next remark examines how the concepts 
 of the Appendix transfer to the Bochner-Lebesgue level.

\begin{rmk}[Induced compatible couple]\label{2.5}
	Let $(X,Y)=(X,Y,Z,e_X,e_Y)$ be a compatible couple 
	(cf.~Definition \ref{7.2}) and $1\leq p,q\leq\infty$. In 
	\cite[Chapter 3, Theorem 1.3]{BS88} it is proved that 
	the sum $R(e_X)+R(e_Y)\subseteq Z$ equipped with the norm
	\begin{align*}
	\|z\|_{R(e_X)+R(e_Y)}:=\inf_{\substack{x\in X,y\in Y\\z=e_Xx+e_Yy}}
	{\max\{\|x\|_X,\|y\|_Y\}}
	\end{align*}
	is a Banach space.
	Then both $e_X:X\rightarrow R(e_X)+R(e_Y)$ and 
	$e_Y:Y\rightarrow R(e_X)+R(e_Y)$ 
	are embeddings (cf.~Definition \ref{7.1}) and therefore due 
	to Proposition \ref{2.4} the induced operators
	\begin{align*}
	\begin{alignedat}{2}
	\boldsymbol{e}_X&:L^p(I,X)\rightarrow L^1(I,R(e_X)+R(e_Y)),&&
	\text{ given via }(\boldsymbol{e}_X\boldsymbol{x})(t):=e_X(\boldsymbol{x}(t))
	\text{ for a.e. }t\in I,\\\boldsymbol{e}_Y&:L^q(I,Y)\rightarrow L^1(I,R(e_X)+R(e_Y)),&&
	\text{ given via }(\boldsymbol{e}_Y\boldsymbol{y})(t):=e_Y(\boldsymbol{y}(t))
\:	\text{ for a.e. }t\in I.
	\end{alignedat}
	\end{align*}
	are embeddings as well. Consequently, the couples
	\begin{align*}
	(L^p(I,X),L^q(I,Y))&=(L^p(I,X),L^q(I,Y),L^1(I,R(e_X)+R(e_Y)),
	\boldsymbol{e}_X,\boldsymbol{e}_Y),\\
	(L^p(I,X),C^0(\overline{I},Y))&=(L^p(I,X),C^0(\overline{I},Y),
	L^1(I,R(e_X)+R(e_Y)),\boldsymbol{e}_X,\boldsymbol{e}_Y\text{id}_{C^0(\overline{I},Y)})
	\end{align*}
	are compatible couples. In accordance with Definition \ref{7.3}, 
	the pull-back intersections
	\begin{align*}
	L^p(I,X)\cap_{\boldsymbol{j}}L^q(I,Y)\qquad\text{ and }\qquad 
	L^p(I,X)\cap_{\boldsymbol{j}}C^0(\overline{I},Y),
	\end{align*}
	where $
	\boldsymbol{j}:=\boldsymbol{e}_Y^{-1}\boldsymbol{e}_X$, and
	 their corresponding intersection embeddings
	\begin{align*}
	\boldsymbol{j}:L^p(I,X)\cap_{\boldsymbol{j}}L^q(I,Y)\rightarrow L^q(I,Y)
	\qquad\text{ and }\qquad\boldsymbol{j}:L^p(I,X)\cap_{\boldsymbol{j}}
	C^0(\overline{I},Y)\rightarrow C^0(\overline{I},Y)
	\end{align*}
	are well-defined.
\end{rmk}

\begin{prop}\label{2.6}
	Let $(X,Y)$ be a compatible couple, $X$ reflexive and $1< p<\infty$. Then for a sequence $(\boldsymbol{x}_n)_{n\in\mathbb{N}}\subseteq L^p(I,X)\cap_{\boldsymbol{j}} C^0(\overline{I},Y)$ and an element $\boldsymbol{x}\in L^p(I,X)\cap_{\boldsymbol{j}} C^0(\overline{I},Y)$ it holds $\boldsymbol{x}_n\overset{n\rightarrow\infty}
	{\rightharpoonup}\boldsymbol{x}$ 
	in $L^p(I,X)\cap_{\boldsymbol{j}} C^0(\overline{I},Y)$ if and only if 
	$(\boldsymbol{x}_n)_{n\in\mathbb{N}}\subseteq L^p(I,X)
	\cap_{\boldsymbol{j}} C^0(\overline{I},Y)$ is bounded and 
	$(\boldsymbol{j}\boldsymbol{x}_n)(t)\overset{n\rightarrow\infty}{\rightharpoonup}
	(\boldsymbol{j}\boldsymbol{x})(t)\text{ in }Y$
	for all $t\in \overline{I}$.
\end{prop}

\begin{proof}
	Immediate consequence of Proposition \ref{2.3} and Proposition \ref{7.5} (iii).\hfill$\square$
\end{proof}

\subsection{Bochner-Sobolev spaces}
Let $(X,\|\cdot\|_X)$ and $(Y,\|\cdot\|_Y)$ be Banach spaces, $j:X\rightarrow Y$ an embedding, 
$I:=\left(0,T\right)$, with $0<T<\infty$, and $1\leq p,q\leq\infty$. A function 
$\boldsymbol{x}\in L^p(I,X)$ has a \textbf{generalized time derivative with respect 
to $j$ in $L^q(I,Y)$} if there exists a function $\boldsymbol{g}\in L^q(I,Y)$ 
such that
	\begin{align*}
	j\left(-\int_I{\boldsymbol{x}(s)\varphi^\prime(s)\;ds}\right)
	=\int_I{\boldsymbol{g}(s)\varphi(s)\;ds}
	\quad\text{ in }Y\quad\text{ for all }\varphi\in C_0^\infty(I).
	\end{align*}
	As such a function $\boldsymbol{g}\in L^q(I,Y)$ is unique 
	(cf.~\cite[Proposition 23.18]{Zei90A}), 
	$\frac{d_j\boldsymbol{x}}{dt}:=\boldsymbol{g}$ is well-defined. By
	\begin{align*}
	W_j^{1,p,q}(I,X,Y):=\bigg\{\boldsymbol{x}\in L^p(I,X)\bigg|\;
	\exists\frac{d_j\boldsymbol{x}}{dt}\in L^q(I,Y)\bigg\}
	\end{align*}
	we denote the \textbf{Bochner-Sobolev space with respect to }$j$,
	 which is equipped with norm
	\begin{align*}
	\|\cdot\|_{W_j^{1,p,q}(I,X,Y)}:=\|\cdot\|_{L^p(I,X)}+\left\|\frac{d_j}{dt}\cdot\right\|_{L^q(I,Y)}
	\end{align*}
	a Banach space (cf.~\cite[Lemma II.5.10]{BF13}). In the case $Y=X$ 
	and $j=\text{id}_X$ we define for sake of readability 
	$\frac{d_X}{dt}:=\frac{d_{\text{id}_X}}{dt}$ and 
	$W^{1,p,q}(I,X):=W_{\text{id}_X}^{1,p,q}(I,X,X)$. 

\begin{prop}\label{2.7} Let $j:X\rightarrow Y$ be an embedding 
	and $\boldsymbol{j}:L^p(I,X)\rightarrow L^p(I,Y)$ 
	given via $(\boldsymbol{j}\boldsymbol{x})(t):=j(\boldsymbol{x}(t))$ in $Y$
	for almost every $t\in I$ and all $\boldsymbol{x}\in L^p(I,X)$. 
	Then it holds $\boldsymbol{x}\in W^{1,p,q}_j(I,X,Y)$ if and only
	 if $\boldsymbol{x}\in L^p(I,X)$ and 
	 $\boldsymbol{j}\boldsymbol{x}\in W^{1,p,q}(I,Y)$. In this case we have
	\begin{align}
		\frac{d_j\boldsymbol{x}}{dt}=\frac{d_Y\boldsymbol{j}\boldsymbol{x}}{dt}
		\quad\text{ in }L^q(I,Y).\label{eq:2.7a}
	\end{align}
\end{prop}

\begin{proof}
A straightforward application of Proposition \ref{2.4} (i).\hfill$\square$
\end{proof}

\begin{prop}\label{2.8}
	Let $A:X\rightarrow Y$ be linear and continuous. Then the induced operator 
	$\mathbfcal{A}:W^{1,p,q}(I,X)\rightarrow W^{1,p,q}(I,Y)$, defined by 
	$(\mathbfcal{A}\boldsymbol{x})(t):=A(\boldsymbol{x}(t))$ in $Y$
	for almost every $t\in I$ and all $\boldsymbol{x}\in W^{1,p,q}(I,X)$,
	 is well-defined, linear and continuous. Furthermore, it holds
	\begin{align}
	\frac{d_Y\mathbfcal{A}\boldsymbol{x}}{dt}=\mathbfcal{A}\frac{d_X\boldsymbol{x}}{dt}\quad\text{ in }L^q(I,Y)\label{eq:kom}
	\end{align}
	for all $\boldsymbol{x}\in W^{1,p,q}(I,X)$. If $A:X\rightarrow Y$ is 
	additionally an isomorphism, then the induced operator 
	$\mathbfcal{A}:W^{1,p,q}(I,X)\rightarrow W^{1,p,q}(I,Y)$ is an isomorphism as well.
\end{prop}

\begin{proof}
	Concerning the well-definedness, linearity, boundedness and \eqref{eq:kom} we refer to \cite[Proposition 2.5.1]{Dro01}. The isomorphism property transfers obviously.
	\hfill $\square$
\end{proof}

\begin{prop}\label{2.9}
	\begin{description}
		\item[(i)] \textbf{First fundamental theorem of calculus for Bochner-Sobolev functions:} 
		Each $\boldsymbol{x}\in W^{1,p,q}(I,X)$ (defined almost everywhere) possesses 
		a unique representation $\{\boldsymbol{x}\}_X\in C^0(\overline{I},X)$ with
		 \begin{align*}
		\{\boldsymbol{x}\}_X(t)=\{\boldsymbol{x}\}_X(t')+
		\int_{t'}^t{\frac{d_X\boldsymbol{x}}{dt}(s)\;ds}\quad\text{ in }X
		\end{align*}
		for all $t',t\in \overline{I}$ with $t'\leq t$. The resulting \textbf{choice function} 
		$\{\cdot\}_X:W^{1,p,q}(I,X)\rightarrow C^0(\overline{I},X)$ is an embedding 
		which we denote by $W^{1,p,q}(I,X)\hookrightarrow C^0(\overline{I},X)$. 
		In consequence, it holds $W^{1,p,q}(I,X)=W^{1,\infty,q}(I,X)$ with norm equivalence. 
		We thus set $W^{1,q}(I,X):=W^{1,\infty,q}(I,X)$.
		\item[(ii)] \textbf{Second fundamental theorem of calculus for Bochner-Sobolev functions:} The operator $\mathbfcal{V}:L^q(I,X)\rightarrow W^{1,q}(I,X)$, given via
		\begin{align*}
		(\mathbfcal{V}\boldsymbol{x})(t):=
		\int_0^t{\boldsymbol{x}(s)\;ds}\quad\text{ in }X\quad\text{ for all }t\in\overline{I}
		\end{align*}
		and every $\boldsymbol{x}\in L^q(I,X)$, is a continuous right inverse of $\frac{d_X}{dt}:W^{1,q}(I,X)\rightarrow L^q(I,X)$.
	\end{description}
\end{prop}
\begin{proof}
	Concerning point (i) we refer to \cite[Lemma II.5.11]{BF13}. Point (ii) except for 
	the continuity one can find in \cite[Kap. IV, Lemma 1.8]{GGZ74}. 
	The verification of the stated continuity is an elementary calculation and thus omitted.
	\hfill $\square$
\end{proof} 

The following result guarantees the compactness, which is indispensable for the applicability of Schauder's 
fixed point theorem, and thus the existence of Galerkin approximations.

\begin{prop}\label{2.10}
	Let $(X,\|\cdot\|_X)$ be a finite dimensional Banach space and $1\leq q\leq\infty$. 
	Then the choice function in Proposition \ref{2.9} (i) is strongly continuous 
	and in the case $1<q\leq \infty$ compact.
\end{prop}

\begin{proof}
	Since for $1<q<\infty$ the space $W^{1,q}(I,X)$ is reflexive and thus strong continuity of $\{\cdot\}_X:W^{1,q}(I,X)\rightarrow C^0(\overline{I},X)$ implies compactness we only prove strong continuity.
	Due to the linearity it suffices to show the strong continuity in the origin 
	$\boldsymbol{0}\in W^{1,q}(I,X)$. To this end, we treat a sequence 
	$(\boldsymbol{x}_n)_{n\in\mathbb{N}}\subseteq W^{1,q}(I,X)$ which 
	weakly converges to $\boldsymbol{0}$ in $W^{1,q}(I,X)$. As the 
	operators $(\boldsymbol{x}\mapsto\{\boldsymbol{x}\}_X(0)):W^{1,q}(I,X)\rightarrow X$ 
	and 
	$\frac{d_X}{dt}:W^{1,q}(I,X)\rightarrow L^q(I,X)$ 
	are weakly continuous we obtain 
	\begin{align*}
	\{\boldsymbol{x}_n\}_X(0)\overset{n\rightarrow\infty}{\rightharpoonup}0
	\quad\text{ in }X\qquad\text{ and }\qquad
	\frac{d_X\boldsymbol{x}_n}{dt}\overset{n\rightarrow\infty}{\rightharpoonup}
	\boldsymbol{0}\quad\text{ in }L^q(I,X).
	\end{align*}
	Thanks to the compactness of $\overline{I}$ there exists a sequence 
	$(t_n)_{n\in\mathbb{N}}\subseteq \overline{I}$ which without loss of 
	generality converges to some $t^*\in\overline{I}$ 
	(otherwise, we switch to a subsequence and use the standard 
	convergence principle \cite[Kap. I, Lemma 5.4]{GGZ74} to obtain the assertion for the entire sequence) 
	such that for all $n\in\mathbb{N}$
	\begin{align*}
	\|\{\boldsymbol{x}_n\}_X\|_{C^0(\overline{I},X)}
	=\max_{t\in\overline{I}}{\|\{\boldsymbol{x}_n\}_X(t)\|_X}=\|\{\boldsymbol{x}_n\}_X(t_n)\|_X.
	\end{align*}
	Next, let us fix an arbitrary $x^*\in X^*$. Using Proposition \ref{2.4} (i), 
	we deduce for every $n\in \mathbb{N}$ that
	\begin{align*}
	\bigg\vert\bigg\langle x^*,\int_0^{t_n}{\frac{d_X\boldsymbol{x}_n}{dt}(s)\;ds}
	\bigg\rangle_{\!\!X}\bigg\vert&=
	\bigg\vert\bigg\langle x^*,\text{sgn}(t_n-t^*)\int_{\min\{t_n,t^*\}}^{\max\{t_n,t^*\}}
	{\frac{d_X\boldsymbol{x}_n}{dt}(s)\;ds}
	+\int_{0}^{t^*}{\frac{d_X\boldsymbol{x}_n}{dt}(s)\;ds}\bigg\rangle_X\bigg\vert\\&
	\leq\|x^*\|_{X^*}\int_{\min\{t_n,t^*\}}^{\max\{t_n,t^*\}}
	{\left\|\frac{d_X\boldsymbol{x}_n}{dt}(s)\right\|_X\!\!\!\!ds}+
	\bigg\vert\left\langle x^*\chi_{\left[0,t^*\right]},
	\frac{d_X\boldsymbol{x}_n}{dt}\right\rangle_{\!\!L^q(I,X)}\bigg\vert.
	\end{align*}
	As $(\frac{d_X\boldsymbol{x}_n}{dt})_{n\in\mathbb{N}}\subseteq L^1(I,X)$ weakly
	 converges to $\boldsymbol{0}$ in $L^1(I,X)$ and is thus uniformly integrable 
	 (cf.~\cite[Theorem 4.2]{BT38}), the right-hand side of the above inequality 
	 tends to zero as $n\rightarrow\infty$.
	This, the integral representation in Proposition \ref{2.9} (i) and the convergence
	 of the initial values yield
	\begin{align*}
	\{\boldsymbol{x}_n\}_X(t_n)=\{\boldsymbol{x}_n\}_X(0)+
	\int_0^{t_n}{\frac{d_X\boldsymbol{x}_n}{dt}(s)\;ds}
	\overset{n\rightarrow\infty}{\rightharpoonup}0\quad\text{ in }X.
	\end{align*} 
    As $X$ is finite dimensional the above convergence is actually strong, and we infer
    \begin{align*}
        \|\{\boldsymbol{x}_n\}_X\|_{C^0(\overline{I},X)}=
        \|\{\boldsymbol{x}_n\}_X(t_n)\|_X
        \overset{n\rightarrow\infty}{\rightarrow}0.\tag*{$\square$}
    \end{align*}
\end{proof}

\subsection{Evolution equations}
Let $(V,\|\cdot\|_V)$ be a reflexive Banach space, 
$(H,(\cdot,\cdot)_H)$ a Hilbert space
 and $j:V\rightarrow H$ an embedding such that $R(j)$ is dense in $H$.
Then the triple $(V,H,j)$ is said to be an \textbf{evolution triple}. 

Denote by $R:H\rightarrow H^*$ the Riesz isomorphism with 
respect to $(\cdot,\cdot)_H$. As $j$ is a dense embedding 
the adjoint $j^*:H^*\rightarrow V^*$ and therefore $e:=j^*Rj:V\rightarrow V^*$
 are embeddings as well. We call $e$ the 
 \textbf{canonical embedding} of $(V,H,j)$. Note that
\begin{align}
\langle ev,w\rangle_V=(jv,jw)_H\quad\text{ for all }v,w\in V.\label{eq:iden}
\end{align} 
For an evolution triple $(V,H,j)$ and $1<p<\infty$ we set
\begin{align*}
	\mathbfcal{X}:=L^p(I,V),\qquad\mathbfcal{W}:=
	W_e^{1,p,p'}(I,V,V^*),\qquad\mathbfcal{Y}:=C^0(\overline{I},H).
\end{align*}

\begin{prop}\label{2.11}
	Let $(V,H,j)$ be an evolution triple and $1< p<\infty$. 
	Then it holds $\boldsymbol{x}\in\mathbfcal{W}$ if and 
	only if $\boldsymbol{x}\in\mathbfcal{X}$ and there exists 
	$\boldsymbol{x}^*\in \mathbfcal{X}^*$ such that
	\begin{align*}
	-\int_I{(j(\boldsymbol{x}(s)),jv)_H\varphi^\prime(s)\;ds}=
	\int_I{\langle\boldsymbol{x}^*(s),v\rangle_V\varphi(s)\;ds}.
	\end{align*}
	for all $v\in V$ and $\varphi\in C_0^\infty(I)$. In this case 
	we have $\frac{d_e\boldsymbol{x}}{dt}=\boldsymbol{x}^*$ 
	in $\mathbfcal{X}^*$.
\end{prop}

\begin{proof}
	If $V$ is additionally separable, a proof can be found in 
	\cite[Proposition 23.20]{Zei90A}. As the argumentation remains true if we omit the separability of $V$, 
	we however refer to this proof.\hfill $\square$
\end{proof}

\begin{prop}\label{2.12}
	Let $(V,H,j)$ be an evolution triple and $1<p<\infty$. Then it holds:
	\begin{description}[{(ii)}]
		\item[{(i)}] Given $\boldsymbol{x}\in \mathbfcal{W}$ the function 
		$\boldsymbol{j}\boldsymbol{x}\in L^p(I,H)$, given via 
		$(\boldsymbol{j}\boldsymbol{x})(t):=j(\boldsymbol{x}(t))$ in $H$ for 
		almost every $t\in I$, possesses a unique representation in 
		$\mathbfcal{Y}$ and the resulting mapping 
		$\boldsymbol{j}:\mathbfcal{W}\rightarrow \mathbfcal{Y}$ is an embedding.
		In particular, the embedding 
		$\mathbfcal{W}\hookrightarrow \mathbfcal{X}\cap_{\boldsymbol{j}}\mathbfcal{Y}$
		 holds true.
		\item[{(ii)}] \textbf{Generalized integration by parts formula:} It holds
		\begin{align*}
		\int_{t'}^{t}{\left\langle
			\frac{d_e\boldsymbol{x}}{dt}(s),\boldsymbol{y}(s)\right\rangle_V\!ds}
		=\left[((\boldsymbol{j}\boldsymbol{x})(s), (\boldsymbol{j}
		\boldsymbol{y})(s))_H\right]^{s=t}_{s=t'}-\int_{t'}^{t}{\left\langle
			\frac{d_e\boldsymbol{y}}{dt}(s),\boldsymbol{x}(s)\right\rangle_V\!ds} 
		\end{align*}
		for all $\boldsymbol{x},\boldsymbol{y}\in \mathbfcal{W}$ 
		and $t,t'\in \overline{I}$ with $t'\leq t$.
	\end{description}  
\end{prop}

\begin{proof}
See \cite[Chapter III.1, Proposition 1.2]{Sho97}.\hfill$\square$
\end{proof}

\begin{defn}[Evolution equation]\label{2.14}
	Let $(V,H,j)$ be an evolution triple and $1<p< \infty$. Furthermore,
	 let $\boldsymbol{y}_0\in H$ be an initial value, 
	 $\boldsymbol{f}\in\mathbfcal{X}^*$ a right-hand side and 
	 $\mathbfcal{A}:\mathbfcal{X}\cap_{\boldsymbol{j}}
	 \mathbfcal{Y}\rightarrow\mathbfcal{X}^*$ 
	 an operator. Then the initial value problem
	\begin{align}
	\begin{split}
	\begin{alignedat}{2}
	\frac{d_e\boldsymbol{y}}{dt}+\mathbfcal{A}\boldsymbol{y}
	&=\boldsymbol{f}&&\quad\text{ in }\mathbfcal{X}^*,\\
	(\boldsymbol{j}\boldsymbol{y})(0)&=\boldsymbol{y}_0&&\quad\text{ in }H
	\end{alignedat}
	\end{split}\label{eq:1}
	\end{align}
	is said to be an \textbf{evolution equation}. The initial condition has to be understood in the sense of the unique continuous representation $\boldsymbol{j}\boldsymbol{y}\in C^0(\overline{I},H)$ (cf. Proposition \ref{2.12} (i)).
\end{defn}

\section{Notions of continuity and growth for evolution equations}
\label{sec:3}

In \cite{KR19} Bochner pseudo-monotonicity and Bochner coercivity has been introduced as appropriate notions of continuity and growth for evolution equations, as they
operate on the same energy space $\mathbfcal{X}\cap_{\boldsymbol{j}}L^\infty(I,H)$ and thus are more on balance than $\frac{d}{dt}$-pseudo-monotonicity and coercivity. 
We emphasize that Bochner pseudo-monotonicity was not directly defined with respect to the weak sequential topology of $\mathbfcal{X}\cap_{\boldsymbol{j}}L^\infty(I,H)$. Indeed,
\cite[Remark 3.5]{Tol18} (see also Remark \ref{rmk} below) gives an example of a sequence $(\boldsymbol{x}_n)_{n\in\mathbb{N}}\subseteq 
\mathbfcal{X}\cap_{\boldsymbol{j}}L^\infty(I,H)$, which satisfies \eqref{eq:1.3} and \eqref{eq:1.4}, but does not weakly converge in $L^\infty(I,H)$, wherefore one cannot simply replace \eqref{eq:1.3} and \eqref{eq:1.4} by weak convergence in $L^\infty(I,H)$ in the definition of Bochner pseudo-monotonicity (cf. Definition \ref{1.3}). However, according to Proposition \ref{2.3}, weak convergence in $\mathbfcal{Y}$ is equivalent to \eqref{eq:1.3} together with \eqref{eq:1.4} valid, not just for almost every, but for all $t\in\overline{I}$. This motivates generalizations of Bochner pseudo-monotonicity, which incorporate the weak sequential topology of $\mathbfcal{X}\cap_{\boldsymbol{j}}\mathbfcal{Y}$.

\begin{defn}[$C^0$-Bochner pseudo-monotonicity 
	and $C^0$-Bochner condition (M)]\label{3.1}
	Let $(V,H,j)$ be an evolution triple and $1<p<\infty$. 
	An operator $\mathbfcal{A}:\mathbfcal{X}\cap_{\boldsymbol{j}}\mathbfcal{Y}
	\rightarrow\mathbfcal{X}^*$ is said to be 
	\begin{description}[{(ii)}]
		\item[{(i)}] \textbf{$C^0$-Bochner pseudo-monotone}, if 
		for a sequence $(\boldsymbol{x}_n)_{n\in\mathbb{N}}\subseteq 
		\mathbfcal{X}\cap_{\boldsymbol{j}}\mathbfcal{Y}$ from
		\begin{align}
		\boldsymbol{x}_n\overset{n\rightarrow\infty}{\rightharpoonup}\boldsymbol{x}
		\quad\text{ in }\mathbfcal{X}\cap_{\boldsymbol{j}}\mathbfcal{Y}&,\label{eq:3.2}\\
		\limsup_{n\rightarrow\infty}{\langle \mathbfcal{A}\boldsymbol{x}_n,
			\boldsymbol{x}_n-\boldsymbol{x}\rangle_{\mathbfcal{X}}}\leq 0\label{eq:3.3}&
		\end{align}
		it follows that $\langle\mathbfcal{A}\boldsymbol{x},
		\boldsymbol{x}-\boldsymbol{y}\rangle_{\mathbfcal{X}}
		\leq	\liminf_{n\rightarrow\infty}{\langle \mathbfcal{A}\boldsymbol{x}_n,
			\boldsymbol{x}_n-\boldsymbol{y}\rangle_{\mathbfcal{X}}}$ for all 
		$\boldsymbol{y}\in \mathbfcal{X}$.
		\item[{(ii)}] satisfying the \textbf{$C^0$-Bochner condition (M)}, if 
		for a sequence $(\boldsymbol{x}_n)_{n\in\mathbb{N}}\subseteq 
		\mathbfcal{X}\cap_{\boldsymbol{j}}\mathbfcal{Y}$ from \eqref{eq:3.2},
		\begin{align}
		\mathbfcal{A}\boldsymbol{x}_n\overset{n\rightarrow\infty}{\rightharpoonup}
		\boldsymbol{\xi}\quad\text{ in }\mathbfcal{X}^*&,\label{eq:3.4}\\
		\limsup_{n\rightarrow\infty}{\langle \mathbfcal{A}\boldsymbol{x}_n,
			\boldsymbol{x}_n\rangle_{\mathbfcal{X}}}\leq \langle\boldsymbol{\xi},
		\boldsymbol{x}\rangle_{\mathbfcal{X}}&\label{eq:3.5}
		\end{align}
		it follows that $\mathbfcal{A}\boldsymbol{x}=\boldsymbol{\xi}$ in $\mathbfcal{X}^*$.
	\end{description}
\end{defn}

\begin{rmk}[$C^0$-Bochner pseudo-monotonicity $\not\Rightarrow$ Bochner pseudo-monotonicity]\label{rmk}\newline
	Clearly, Bochner pseudo-monotonicity implies $C^0$-Bochner pseudo-monotonicity. This is an immediate consequence of Proposition \ref{2.3}.
	Note that the converse is not true in general. In fact, there exist $C^0$-Bochner pseudo-monotone operators which are not Bochner pseudo-monotone. This can be seen by the following 
    example (cf.~\cite[Remark 3.5]{Tol18}). 
    
    Let $I=\left(-1,1\right)$, $p\in\left(1,\infty\right)$, $V=H=\mathbb{R}$ and 
    $\mathbfcal{A}:L^\infty(I,\mathbb{R})\to L^{p'}(I,\mathbb{R})$ 
    given via $\mathbfcal{A}\boldsymbol{x}:=\langle\boldsymbol{\omega},
    \boldsymbol{x}\rangle_{L^\infty(I,\mathbb{R})}
    =\int_I{\boldsymbol{x}(s)d\boldsymbol{\omega}(s)}$
	 for all $\boldsymbol{x}\in L^\infty(I,\mathbb{R})$, where 
	 $\boldsymbol{\omega}\in (L^\infty(I,\mathbb{R}))^*$ is a finitely additive 
	 measure with $\boldsymbol{\omega}(\left(-1/2n,0\right)\cup\left(0,1/2n\right))=1$ 
	 for all $n\in \mathbb{N}$, whose existence is guaranteed in \cite[Theorem 2.9]{Tol18}. We define 
	 $(\boldsymbol{x}_n)_{n\in \mathbb{N}}\subseteq L^\infty(I,\mathbb{R})$ 
	 by $\boldsymbol{x}_n(0):=0$, $\boldsymbol{x}_n(t):=0$ if 
	 $\vert t\vert\ge 2/n$, $\boldsymbol{x}_n(t):=1$ if $0< \vert t\vert\leq 1/n$, 
	 and $\boldsymbol{x}_n(t):=-n \vert t\vert +2$ if $1/n<\vert t\vert< 2/n$. One easily sees, that $(\boldsymbol{x}_n)_{n\in \mathbb{N}}\subseteq L^\infty(I,\mathbb{R})$ with $\sup_{n\in\mathbb{N}}\|\boldsymbol{x}_n\|_{L^\infty(I,\mathbb{R})}\leq 1$ and $\boldsymbol{x}_n(t)\overset{n\to\infty}{\to }0$ for every $t\in I$, which immediately implies that $\boldsymbol{x}_n\overset{\ast}{\rightharpoondown}\boldsymbol{0}\text{ in }L^\infty(I,\mathbb{R})$ $(n\to\infty)$ and $\boldsymbol{x}_n\overset{n\to \infty}{\to}\boldsymbol{0} \text{ in }L^p(I,\mathbb{R})$. Apart from that, according to \cite[Theorem 2.8]{Tol18}, we have $\mathbfcal{A}\boldsymbol{x}_n=\langle\boldsymbol{\omega},
	 \boldsymbol{x}_n\rangle_{L^\infty(I,\mathbb{R})}=1$ for all $n\in \mathbb{N}$, which let us exclude that $\boldsymbol{x}_n\not\rightharpoonup \boldsymbol{0}$ in 
	 $L^\infty(I,\mathbb{R})$ $(n\to\infty)$ and provides that $ \limsup_{n\rightarrow\infty}{\langle\mathbfcal{A}\boldsymbol{x}_n,
	 	\boldsymbol{x}_n-\boldsymbol{0}\rangle_{L^p(I,\mathbb{R})}}= \lim_{n\rightarrow\infty}{\int_I{\boldsymbol{x}_n(s)ds}}=0$.
	 Overall, $(\boldsymbol{x}_n)_{n\in \mathbb{N}}\subseteq L^\infty(I,\mathbb{R})$ satisfies \eqref{eq:1.2}--\eqref{eq:1.5}, but $\liminf_{n\rightarrow\infty}{\langle\mathbfcal{A}\boldsymbol{x}_n,
	\boldsymbol{x}_n-\boldsymbol{1}\rangle_{L^p(I,\mathbb{R})}}=-2<0=\langle\mathbfcal{A}\boldsymbol{0},
\boldsymbol{0}-\boldsymbol{1}\rangle_{L^p(I,\mathbb{R})}$, i.e.,
 $\mathbfcal{A}:L^\infty(I,\mathbb{R})\to L^{p'}(I,\mathbb{R})$ cannot be Bochner 
 pseudo-monotone. However, if $(\boldsymbol{x}_n)_{n\in\mathbb{N}}\subseteq C^0(\overline{I},\mathbb{R})$ is a sequence satisfying \eqref{eq:3.2} and \eqref{eq:3.3}, then $\mathbfcal{A}\boldsymbol{x}_n\overset{n\to\infty}{\to }\mathbfcal{A}\boldsymbol{x}$ in $\mathbb{R}$, as also $\boldsymbol{\omega}\in (C^0(\overline{I},\mathbb{R}))^*$ and therefore
 $\langle\mathbfcal{A}\boldsymbol{x},
 \boldsymbol{x}-\boldsymbol{y}\rangle_{L^p(I,\mathbb{R})}\leq \liminf_{n\rightarrow\infty}{\langle\mathbfcal{A}\boldsymbol{x}_n,
 	\boldsymbol{x}_n-\boldsymbol{y}\rangle_{L^p(I,\mathbb{R})}}$ for any $\boldsymbol{y}\in L^p(I,\mathbb{R})$. In other words, 
 $\mathbfcal{A}:L^\infty(I,\mathbb{R})\to L^{p'}(I,\mathbb{R})$ is $C^0$-Bochner pseudo-monotone.
\end{rmk}

\begin{prop}\label{3.6}
	Let $(V,H,j)$ be an evolution triple and $1<p<\infty$. Then it holds:
	\begin{description}[{(ii)}]
        \item[{(i)}] If $\mathbfcal{A}:\mathbfcal{X}\cap_{\boldsymbol{j}}\mathbfcal{Y}
        \rightarrow\mathbfcal{X}^*$ is $C^0$-Bochner pseudo-monotone, 
        then it satisfies the $C^0$-Bochner condition (M).
		\item[{(ii)}] If $\mathbfcal{A}:\mathbfcal{X}\cap_{\boldsymbol{j}}\mathbfcal{Y}
		\rightarrow\mathbfcal{X}^*$ is locally bounded and satisfies 
		the $C^0$-Bochner condition (M), then it is demi-continuous.
	\end{description}
\end{prop}

\begin{proof}
\textbf{ad (i)} Let $(\boldsymbol{x}_n)_{n\in\mathbb{N}}\subseteq
 \mathbfcal{X}\cap_{\boldsymbol{j}}\mathbfcal{Y}$ be a sequence 
 satisfying \eqref{eq:3.2}, \eqref{eq:3.4} and \eqref{eq:3.5}.
  In particular, \eqref{eq:3.4} and \eqref{eq:3.5} imply \eqref{eq:3.3}. 
  The $C^0$-Bochner pseudo-monotonicity of $\mathbfcal{A}:
  \mathbfcal{X}\cap_{\boldsymbol{j}}\mathbfcal{Y}\rightarrow
  \mathbfcal{X}^*$, \eqref{eq:3.3} and \eqref{eq:3.4} thus imply
\begin{align*}
\langle\mathbfcal{A}\boldsymbol{x},\boldsymbol{x}-\boldsymbol{y}\rangle_{\mathbfcal{X}}
&\leq \liminf_{n\rightarrow\infty}{\langle\mathbfcal{A}\boldsymbol{x}_n,
	\boldsymbol{x}_n-\boldsymbol{y}\rangle_{\mathbfcal{X}}}\\&
\leq \limsup_{n\rightarrow\infty}{\langle\mathbfcal{A}\boldsymbol{x}_n,
	\boldsymbol{x}_n-\boldsymbol{x}\rangle_{\mathbfcal{X}}}+
\limsup_{n\rightarrow\infty}{\langle\mathbfcal{A}\boldsymbol{x}_n,
	\boldsymbol{x}-\boldsymbol{y}\rangle_{\mathbfcal{X}}}\leq
 \langle\boldsymbol{\xi},\boldsymbol{x}-\boldsymbol{y}\rangle_{\mathbfcal{X}}
\end{align*}
for all $\boldsymbol{y}\in \mathbfcal{X}$ and therefore 
$\mathbfcal{A}\boldsymbol{x}=\boldsymbol{\xi}$ in $\mathbfcal{X}^*$.

\textbf{ad (ii)} Let $(\boldsymbol{x}_n)_{n\in\mathbb{N}}\subseteq 
\mathbfcal{X}\cap_{\boldsymbol{j}}\mathbfcal{Y}$ be a sequence 
such that $\boldsymbol{x}_n\overset{n\rightarrow\infty}{\rightarrow}\boldsymbol{x}$
 in $\mathbfcal{X}\cap_{\boldsymbol{j}}\mathbfcal{Y}$. From the locally
 boundedness of
$\mathbfcal{A}:\mathbfcal{X}\cap_{\boldsymbol{j}}\mathbfcal{Y}\rightarrow\mathbfcal{X}^*$ 
and reflexivity of $\mathbfcal{X}^*$
we obtain a subsequence
$(\mathbfcal{A}\boldsymbol{x}_n)_{n\in\Lambda}\subseteq
\mathbfcal{X}^*$, with $\Lambda\subseteq\mathbb{N}$, and
$\boldsymbol{\xi}\in \mathbfcal{X}^*$ such that
$\mathbfcal{A}\boldsymbol{x}_n\overset{n\rightarrow\infty}
{\rightharpoonup}\boldsymbol{\xi}$
in $\mathbfcal{X}^*$ $(n\in\Lambda)$. Hence, it holds 
$\langle\mathbfcal{A}\boldsymbol{x}_n,\boldsymbol{x}_n\rangle_{\mathbfcal{X}}
\overset{n\rightarrow\infty}{\rightarrow}\langle \boldsymbol{\xi},
\boldsymbol{x}\rangle_{\mathbfcal{X}}$ $(n\in \Lambda)$, i.e., \eqref{eq:3.5} 
with respect to $\Lambda$. From the $C^0$-Bochner condition (M) we conclude 
$\mathbfcal{A}\boldsymbol{x}=\boldsymbol{\xi}$ in $\mathbfcal{X}^*$. 
As this argumentation stays valid for each subsequence of 
$(\boldsymbol{x}_n)_{n\in\mathbb{N}}\subseteq\mathbfcal{X}\cap_{\boldsymbol{j}}\mathbfcal{Y}$, 
$\mathbfcal{A}\boldsymbol{x}\in\mathbfcal{X}^*$ is weak accumulation 
point of each subsequence of 
${(\mathbfcal{A}\boldsymbol{x}_n)_{n\in\mathbb{N}}\subseteq\mathbfcal{X}^*}$. The 
standard convergence principle (cf.~\cite[Kap. I, Lemma 5.4]{GGZ74}) finally yields 
$\mathbfcal{A}\boldsymbol{x}_n\overset{n\rightarrow\infty}{\rightharpoonup} 
\mathbfcal{A}\boldsymbol{x}\text{ in }\mathbfcal{X}^* $.\hfill$\square$
\end{proof}

\begin{rmk}[$C^0$-Bochner condition (M) $\not\Rightarrow$ $C^0$-Bochner pseudo-monotonicity]\label{rmk2}\newline
	According to Proposition \ref{3.6} (i) $C^0$-Bochner pseudo-monotonicity implies the $C^0$-Bochner condition (M). But note that there exists operators satisfying the $C^0$-Bochner condition (M) without being $C^0$-Bochner pseudo-monotone. 
	
	For example, let $I=\left(0,T\right)$, with $0<T<\infty$, $p\in \left(1,\infty\right)$, $V=H$ a separable Hilbert space with orthonormal basis $(e_n)_{n\in\mathbb{N}}\subseteq H$ and Riesz isomorphism $R:H\to H^*$. Moreover, let $\mathbfcal{A}:C^0(\overline{I},H)\to L^{p'}(I,H^*)$ be given via $(\mathbfcal{A}\boldsymbol{x})(t)=-R(\boldsymbol{x}(t))$ in $H^*$ for almost every $t\in I$ and all $\boldsymbol{x}\in C^0(\overline{I},H)$. Then, $\mathbfcal{A}:C^0(\overline{I},H)\to L^{p'}(I,H^*)$ satisfies the $C^0$-Bochner condition (M), as it is weakly continuous, but is not $C^0$-Bochner pseudo-monotone. In fact, the sequence $(\boldsymbol{x}_n)_{n\in\mathbb{N}}\subseteq C^0(\overline{I},H)$, given via $\boldsymbol{x}_n\equiv e_n$ for every $n\in \mathbb{N}$, satisfies $\boldsymbol{x}_n\overset{n\to\infty}{\rightharpoonup}\boldsymbol{0}$ in $C^0(\overline{I},H)$ and $\limsup_{n\to\infty}{\langle\mathbfcal{A}\boldsymbol{x}_n,\boldsymbol{x}_n-\boldsymbol{0}\rangle_{ L^p(I,H)}}=-T<0$, but $\liminf_{n\to\infty}{\langle\mathbfcal{A}\boldsymbol{x}_n,\boldsymbol{x}_n-\boldsymbol{y}\rangle_{ L^p(I,H)}}=-T<0=\langle\mathbfcal{A}\boldsymbol{0},\boldsymbol{0}-\boldsymbol{y}\rangle_{ L^p(I,H)}$ for any $\boldsymbol{y}\in L^p(I,H)$. 
\end{rmk}

\begin{defn}[$C^0$-Bochner coercivity]\label{3.7}
	Let $(V,H,j)$ be an evolution triple and $1<p<\infty$.
	An operator $\mathbfcal{A}:\mathbfcal{X}\cap_{\boldsymbol{j}}\mathbfcal{Y}\rightarrow\mathbfcal{X}^*$
	 is said to be
	\begin{description}[{(ii)}]
		\item[{(i)}] \textbf{$C^0$-Bochner coercive with respect to $\boldsymbol{f}\in\mathbfcal{X}^*$ and $\boldsymbol{x}_0\in H$}, 
		if there exists a constant $M:=M(\boldsymbol{f},\boldsymbol{x}_0,\mathbfcal{A})>0$ 
		such that for all $\boldsymbol{x}\in\mathbfcal{X}\cap_{\boldsymbol{j}}\mathbfcal{Y}$ 
		from
		\begin{align}
		\frac{1}{2}\|(\boldsymbol{j}\boldsymbol{x})(t)\|_H^2+\langle\mathbfcal{A}\boldsymbol{x}-
		\boldsymbol{f},\boldsymbol{x}\chi_{\left[0,t\right]}\rangle_{\mathbfcal{X}}\leq \frac{1}{2}\|\boldsymbol{x}_0\|_H^2\quad
		\text{ for all }t\in \overline{I}\label{eq:bcoer}
		\end{align}
		it follows that $\|\boldsymbol{x}\|_{\mathbfcal{X}\cap_{\boldsymbol{j}}\mathbfcal{Y}}\leq M$.
		\item[{(ii)}] \textbf{$C^0$-Bochner coercive}, if it is $C^0$-Bochner coercive with 
		respect to $\boldsymbol{f}$ and $\boldsymbol{x}_0$ for all $\boldsymbol{f}\in\mathbfcal{X}^*$ and $\boldsymbol{x}_0\in H$.
	\end{description}
\end{defn}

Note that $C^0$-Bochner coercivity, similar to semi-coercivity
(cf.~\cite{Rou05}) in conjunction with Gronwall's inequality, takes
into account the information from the operator and the time
derivative. In fact, $C^0$-Bochner coercivity is a more general property. In
the context of the main theorem on pseudo-monotone perturbations of
maximal monotone mappings (cf.~\cite[§32.4.]{Zei90B}), which implies
Theorem \ref{1.2}, $C^0$-Bochner coercivity is phrased in the spirit of a
local coercivity\footnote{$A:D(A)\subseteq V\to V^*$ is said to be
	coercive (cf.~\cite[§32.4.]{Zei90B}) with respect to $f\in V^*$, if
	$D(A)$ is unbounded and there exists a constant $R>0$, such that for
	$v\in V$ from $\langle Av,v\rangle_V\leq \langle f,v\rangle_V$ it
	follows $\|v\|_V\leq R$, i.e., all elements whose images with
	respect to $A$ do not grow beyond the data $f$ in this weak sense
	are contained in a fixed ball in $V$.} type condition of
$\frac{d_e}{dt}+\mathbfcal{A}:\mathbfcal{W}\subseteq \mathbfcal{X}\to
\mathbfcal{X}^*$. Being more precise, if
$\mathbfcal{A}:\mathbfcal{X}\cap_{\boldsymbol{j}}\mathbfcal{Y}\to
\mathbfcal{X}^*$ is $C^0$-Bochner coercive with respect to
$\boldsymbol{f}\in\mathbfcal{X}^*$ and $\boldsymbol{x}_0\in H$, then
for $\boldsymbol{x}\in \mathbfcal{W}$ from
${\|(\boldsymbol{j}\boldsymbol{x})(0)\|_H\leq
	\|\boldsymbol{x}_0\|_H}$, i.e., 
$\langle
\frac{d_e\boldsymbol{x}}{dt},\boldsymbol{x}\rangle_{\mathbfcal{X}}\ge
-\frac{1}{2}\|\boldsymbol{x}_0\|_H^2$, and
\begin{align}
\bigg\langle \frac{d_e\boldsymbol{x}}{dt}+\mathbfcal{A}\boldsymbol{x},\boldsymbol{x}\chi_{\left[0,t\right]}\bigg\rangle_{\mathbfcal{X}}\leq \langle \boldsymbol{f},\boldsymbol{x}\chi_{\left[0,t\right]}\rangle_{\mathbfcal{X}}\quad\text{ for all }t\in \overline{I}.\label{eq:bcoer2}
\end{align}
it follows $\|\boldsymbol{x}\|_{\mathbfcal{X}\cap_{\boldsymbol{j}}\mathbfcal{Y}}\leq M$, since \eqref{eq:bcoer2} is just \eqref{eq:bcoer}. In other words, if the image of $\boldsymbol{x}\in\mathbfcal{W}$ with respect to $\frac{d_e}{dt}$ and $\mathbfcal{A}$ is bounded by the data $\boldsymbol{x}_0$, $\boldsymbol{f}$ in this weak sense, then $\boldsymbol{x}$ is contained in a fixed ball in $\mathbfcal{X}\cap_{\boldsymbol{j}}\mathbfcal{Y}$. We chose \eqref{eq:bcoer} instead of \eqref{eq:bcoer2} in Definition \ref{3.7}, since $\boldsymbol{x}\in \mathbfcal{X}\cap_{\boldsymbol{j}}\mathbfcal{Y}$ is not admissible in \eqref{eq:bcoer2}.

We emphasize that there is a relation between $C^0$-Bochner coercivity and coercivity in the sense of Definition \ref{2.1}. In fact, in the case of
bounded operators
$\mathbfcal{A}:\mathbfcal{X}\rightarrow \mathbfcal{X}^*$, $C^0$-Bochner
coercivity extends the standard concept of coercivity.

\begin{prop}\label{3.8}
	Let $(V,H,j)$ be an evolution triple and $1<p<\infty$. If 
	$\mathbfcal{A}:D(\mathbfcal{A})\subseteq\mathbfcal{X}\rightarrow \mathbfcal{X}^*$ 
	with $D(\mathbfcal{A})= \mathbfcal{X}\cap_{\boldsymbol{j}}\mathbfcal{Y}$ is bounded 
	and coercive (in the sense of Definition \ref{2.1}), then $\mathbfcal{A}:\mathbfcal{X}\cap_{\boldsymbol{j}}
	\mathbfcal{Y}\rightarrow \mathbfcal{X}^*$ is $C^0$-Bochner coercive.
\end{prop}

\begin{proof}
	A straightforward adaptation of \cite[Lemma 3.21]{KR19}.\hfill$\square$
\end{proof}

\begin{lem}[Induced Bochner pseudo-monotonicity and Bochner coercivity]\label{HL}
	\upshape Let $(V,H,j)$ be an evolution triple, $1<p<\infty$ and
	$A(t):V\to V^*$, $t\in I$, a family of operators with the following properties:
	\begin{description}
		\item[\textbf{(C.1)}] \hypertarget{C.1}
		$A(t):V\rightarrow V^*$ is pseudo-monotone for
		almost every $t\in I$.
		\item[\textbf{(C.2)}] \hypertarget{C.2}
		$A(\cdot)v:I\rightarrow V^*$ is Bochner measurable for
		all $v\in V$.
		\item[\textbf{(C.3)}] \hypertarget{C.3} For some non-negative
		functions $\alpha,\gamma\in L^{p'}(I)$, $\beta\in L^\infty(I)$ and
		a non-decreasing function
		$\mathscr{B}:\mathbb{R}_{\ge 0}\rightarrow \mathbb{R}_{\ge 0}$
		holds
		\begin{align*}
		\left\|A(t)v\right\|_{V^*}\leq \mathscr{B}(\|jv\|_H)(\alpha(t)+\beta(t)\|v\|_{V}^{p-1})+\gamma(t)
		\end{align*}
		for almost every $t\in I$ and all $v\in V$.
		\item[\textbf{(C.4)}] \hypertarget{C.4}{} For some constant $c_0>0$ and
		non-negative functions $c_1,c_2\in L^1(I)$ holds 
		\begin{align*}
		\langle A(t)v,v\rangle_V\ge c_0\|v\|_V^p-c_1(t)\|jv\|_H^2-c_2(t)
		\end{align*}
		for almost every $t\in I$ and all $v\in V$.
	\end{description}
	Then the induced operator
	$\mathbfcal{A}:\mathbfcal{X}\cap_{\boldsymbol{j}}L^\infty(I,H)\rightarrow\mathbfcal{X}^*$, given via $(\mathbfcal{A}\boldsymbol{x})(t):=A(t)(\boldsymbol{x}(t))$ in $V^*$ for almost every $t\in I$ and all $\boldsymbol{x}\in \mathbfcal{X}\cap_{\boldsymbol{j}}L^\infty(I,H)$,
	is well-defined, bounded, Bochner pseudo-monotone and Bochner coercive.
\end{lem}

\begin{proof}
	If $V$ is additionally separable, a proof can be found in \cite[Proposition 3.13]{KR19}. As the argumentation remains true if we omit the separability of $V$, we however refer to this proof.\hfill$\square$
\end{proof}

\section{Abstract Galerkin approach}
\label{sec:4}
In this section we specify the exact framework of the implemented Galerkin approach.
\begin{rmk}[Galerkin approximation]\label{4.1} Let $(V,H,j)$ be an evolution triple and 
	$1<p< \infty$. Furthermore, let $\boldsymbol{y}_0\in H$, $\boldsymbol{f}\in\mathbfcal{X}^*$
	 and $\mathbfcal{A}:\mathbfcal{X}\cap_{\boldsymbol{j}}\mathbfcal{Y}\rightarrow\mathbfcal{X}^*$.
	\begin{description}[{(iii)}]
		\item[{(i)}]\textbf{Galerkin-Basis:} Let $\mathcal{U}$ be a system of subspaces 
		$U\subseteq V$ such that $(U,\|\cdot\|_V)$ is a Banach space and 
		\begin{align*}
		\overline{\bigcup_{U\in\;\mathcal{U}}{U}}^{\|\cdot\|_V}=V.
		\end{align*} 
		Moreover, for $U\in \mathcal{U}$ we set $H_U:=\overline{j(U)}^{\|\cdot\|_H}$
		and by $\boldsymbol{y}_0^{U}\in H_U$  we denote a net of approximative  initial values
		 such that
		\begin{align*}
		\boldsymbol{y}_0^{U}\overset{U\in\mathcal{U}}{\rightarrow}\boldsymbol{y}_0\quad\text{ in }H
		\qquad\text{ and }\qquad\|\boldsymbol{y}_0^{U}\|_H\leq\|\boldsymbol{y}_0\|_H
	\quad	\text{ for all }U\in\mathcal{U},
		\end{align*}
		where the notion of convergence initially has to be understood in the sense of 
		nets but will later be realized by sequential convergence. Such a family 
		$(U,\boldsymbol{y}_0^{U})_{U\in\;\mathcal{U}}$ is called a \textbf{Galerkin basis} 
		of $(V,\boldsymbol{y}_0)$.
		\item[{(ii)}]\textbf{Restriction of evolution triple structure:} $(H_U,(\cdot,\cdot)_H)$ 
		is a Hilbert space and the restricted operator  $j_U:=j_{|_U}:U\rightarrow H_U$
		a dense embedding, i.e., $(U,H_U,j_U)$ forms an evolution triple. In particular, 
		the corresponding canonical embedding $e_U:U\rightarrow U^*$ satisfies
		\begin{align}
		\langle e_Uu,\tilde{u}\rangle_U=(ju,j\tilde{u})_H=\langle eu,\tilde{u}\rangle_V\quad
		\text{ for all }u,\tilde{u}\in U.\label{eq:4.2}
		\end{align} 
		\item[{(iii)}]\textbf{Restriction of energy spaces:} For $U\in\mathcal{U}$ and 
		$I:=\left(0,T\right)$, with $0<T<\infty$, we set
		\begin{align*}
			\mathbfcal{X}_U:=L^p(I,U),\qquad\mathbfcal{W}_U:=W_{e_U}^{1,p,p'}(I,U,U^*),
			\qquad \mathbfcal{Y}_U:=C^0(\overline{I},H_U).
		\end{align*} Due to Remark \ref{2.5} the couple $(\mathbfcal{X}_U,\mathbfcal{Y}_U)
		:=(\mathbfcal{X}_U,\mathbfcal{Y}_U,L^1(I,H_U),\boldsymbol{j}_U,\text{id}_{\mathbfcal{Y}_U})$ 
		forms a compatible couple, where $\boldsymbol{j}_U:\mathbfcal{X}_U
		\cap_{\boldsymbol{j}_U}\mathbfcal{Y}_U\rightarrow \mathbfcal{Y}_U$ is given via 
		$(\boldsymbol{j}_U\boldsymbol{x})(t):=j_U(\boldsymbol{x}(t))=(\boldsymbol{j}\boldsymbol{x})(t)$ 
		for almost every $t\in I$ and all $\boldsymbol{x}\in 
		\mathbfcal{X}_U\cap_{\boldsymbol{j}_U}\mathbfcal{Y}_U$ (cf.~Proposition \ref{2.4}). 
		In particular, it holds $\mathbfcal{X}_U\cap_{\boldsymbol{j}_U}\mathbfcal{Y}_U
		\hookrightarrow \mathbfcal{X}\cap_{\boldsymbol{j}}\mathbfcal{Y}$. Apart from this, 
		Proposition \ref{2.12} provides the embedding $\mathbfcal{W}_{U}\hookrightarrow 
		\mathbfcal{X}_U\cap_{\boldsymbol{j}_U}\mathbfcal{Y}_U$ and the generalized 
		integration by parts formula with respect to $U$:
			\begin{align}
		\int_{t'}^{t}{\left\langle
			\frac{d_{e_U}\boldsymbol{x}}{dt}(s),\boldsymbol{y}(s)\right\rangle_U\!ds}
		=\left[((\boldsymbol{j}\boldsymbol{x})(s), (\boldsymbol{j}
		\boldsymbol{y})(s))_H\right]^{s=t}_{s=t'}-\int_{t'}^{t}{\left\langle
			\frac{d_{e_U}\boldsymbol{y}}{dt}(s),\boldsymbol{x}(s)\right\rangle_U\!ds} \label{eq:4.3}
		\end{align}
		for all $\boldsymbol{x},\boldsymbol{y}\in \mathbfcal{W}_U$ and $t,t'\in \overline{I}$
		 with $t'\leq t$.
		\item[{(iv)}]\textbf{Restriction of operators:} For $U\in\mathcal{U}$ we define the 
		restricted operator and right-hand side by
		\begin{align*}
		\mathbfcal{A}_U:=(\text{id}_{\mathbfcal{X}_U})^*\mathbfcal{A}: {\mathbfcal{X}_U}
		\cap_{\boldsymbol{j}_U}{\mathbfcal{Y}_U}\rightarrow\mathbfcal{X}_U^*\qquad
		\text{ and }\qquad\boldsymbol{f}_U:=\left(\text{id}_{\mathbfcal{X}_U}\right)^*
		\boldsymbol{f}\in\mathbfcal{X}_U^*.
		\end{align*}
		Then it holds for all $\boldsymbol{x}\in {\mathbfcal{X}_U}\cap_{\boldsymbol{j}_U}{\mathbfcal{Y}_U}$
		and $\tilde{\boldsymbol{x}}\in \mathbfcal{X}_U$
		\begin{align}
		\langle\mathbfcal{A}_U\boldsymbol{x},\tilde{\boldsymbol{x}}\rangle_{\mathbfcal{X}_U}
		=\langle \mathbfcal{A}\boldsymbol{x},\tilde{\boldsymbol{x}}\rangle_{\mathbfcal{X}}
		\qquad\text{ and }\qquad\langle\boldsymbol{f}_U,\tilde{\boldsymbol{x}}
		\rangle_{\mathbfcal{X}_U}=\langle \boldsymbol{f},\tilde{\boldsymbol{x}}
		\rangle_{\mathbfcal{X}}.\label{eq:4.4}
		\end{align}
	.
		\item[{(v)}]\textbf{Galerkin system:} Given a Galerkin basis 
		$(U,\boldsymbol{y}^{U}_0)_{U\in\;\mathcal{U}}$ we obtain the 
		well-posedness of the system of evolution equations
		\begin{align*}
		\begin{split}
		\begin{alignedat}{2}
		\frac{d_{e_U}\boldsymbol{y}_U}{dt}+\mathbfcal{A}_U\boldsymbol{y}_U&
		=\boldsymbol{f}_U&&\quad\text{ in }\mathbfcal{X}_U^*,\\
		(\boldsymbol{j}_U\boldsymbol{y}_U)(0)&=\boldsymbol{y}^{U}_0&&\quad\text{ in } 
		H_U, \quad U\in \mathcal{U}.
		\end{alignedat}
		\end{split}
		\end{align*}
		Such a system is called \textbf{Galerkin system with respect to} 
		$(U,\boldsymbol{y}^{U}_0)_{U\in\;\mathcal{U}}$.
	\end{description}
\end{rmk}

The next lemma examines to what extent the properties of the global operator, 
especially those developed in Section \ref{sec:3}, transfer to its restriction as per 
Remark \ref{4.1} (iv).

\begin{lem}\label{4.5} Let $(V,H,j)$ be an evolution triple and $1<p< \infty$. 
	Furthermore, let $\boldsymbol{y}_0\in H$, $\boldsymbol{f}\in\mathbfcal{X}^*$ 
	and $\mathbfcal{A}:\mathbfcal{X}\cap_{\boldsymbol{j}}\mathbfcal{Y}
	\rightarrow\mathbfcal{X}^*$. 
If $(U,\boldsymbol{y}^{U}_0)_{U\in\;\mathcal{U}}$ is a Galerkin basis of 
$(V,\boldsymbol{y}_0)$, $\mathbfcal{A}_U:=(\text{id}_{\mathbfcal{X}_U})^*\mathbfcal{A}:
{\mathbfcal{X}_U}\cap_{\boldsymbol{j}_U}{\mathbfcal{Y}_U}\rightarrow\mathbfcal{X}_U^*$
 and $\boldsymbol{f}_U:=(\text{id}_{\mathbfcal{X}_U})^*\boldsymbol{f}\in \mathbfcal{X}_U^*$, 
 then it holds:
	\begin{description}[(iii)]
		\item[(i)] If $\mathbfcal{A}: \mathbfcal{X}\cap_{\boldsymbol{j}}\mathbfcal{Y}
		\rightarrow\mathbfcal{X}^*$
		 is bounded, demi-continuous or $C^0$-Bochner pseudo-monotone, then 
		 $\mathbfcal{A}_U: {\mathbfcal{X}_U}\cap_{\boldsymbol{j}_U}{\mathbfcal{Y}_U}
		 \rightarrow\mathbfcal{X}_U^*$ is as well.
		\item[(ii)]  If $\mathbfcal{A}: \mathbfcal{X}\cap_{\boldsymbol{j}}\mathbfcal{Y}
		\rightarrow\mathbfcal{X}^*$
		 is $C^0$-Bochner coercive with respect to $\boldsymbol{f}\in \mathbfcal{X}^*$ and $\boldsymbol{y}_0\in H$, then 
		 $\mathbfcal{A}_U: {\mathbfcal{X}_U}\cap_{\boldsymbol{j}_U}{\mathbfcal{Y}_U}
		 \rightarrow\mathbfcal{X}_U^*$ is $C^0$-Bochner coercive with respect to 
		 $\boldsymbol{f}_U\in \mathbfcal{X}_U^*$ and $\boldsymbol{y}_0^{U}\in H_U$.
		\item[(iii)] If $\mathbfcal{A}_0: \mathbfcal{X}
		\rightarrow\mathbfcal{X}^*$ is monotone, 
		$\mathbfcal{B}: \mathbfcal{X}\cap_{\boldsymbol{j}}\mathbfcal{Y}\rightarrow\mathbfcal{X}^*$ 
		is bounded and $\mathbfcal{A}:=\mathbfcal{A}_0+\mathbfcal{B}: 
		\mathbfcal{X}\cap_{\boldsymbol{j}}\mathbfcal{Y}\rightarrow\mathbfcal{X}^*$
		 satisfies the $C^0$-Bochner condition (M), then $\mathbfcal{A}_U:
		 {\mathbfcal{X}_U}\cap_{\boldsymbol{j}_U}{\mathbfcal{Y}_U}
		 \rightarrow\mathbfcal{X}_U^*$ satisfies the $C^0$-Bochner condition (M).
	\end{description}
\end{lem}

\begin{proof} 
	\textbf{ad (i)/(ii)} (i) follows from
	the embedding ${\mathbfcal{X}_U}\cap_{\boldsymbol{j}_U}{\mathbfcal{Y}_U}
	\hookrightarrow \mathbfcal{X}\cap_{\boldsymbol{j}}\mathbfcal{Y}$, 
	the weak continuity of $(\text{id}_{\mathbfcal{X}_U})^*:\mathbfcal{X}^*
	\rightarrow\mathbfcal{X}_U^*$ and the identities \eqref{eq:4.4}. (ii) follows from \eqref{eq:4.4}, $\|\boldsymbol{y}_0^{U}\|_{H_U}\leq \|\boldsymbol{y}_0\|_H$, $\|\cdot\|_{{\mathbfcal{X}_U}\cap_{\boldsymbol{j}_U}{\mathbfcal{Y}_U}}=\|\cdot\|_{{\mathbfcal{X}}\cap_{\boldsymbol{j}}{\mathbfcal{Y}}}$ on ${\mathbfcal{X}}\cap_{\boldsymbol{j}}{\mathbfcal{Y}}$ and $\|\cdot\|_{H_U}=\|\cdot\|_H$ on $H_U$.
	
	\textbf{ad (iii)} Let $(\boldsymbol{x}_n)_{n\in\mathbb{N}}
\subseteq {\mathbfcal{X}_U}\cap_{\boldsymbol{j}_U}{\mathbfcal{Y}_U}$ 
satisfy \eqref{eq:3.2}--\eqref{eq:3.5} with respect to $\mathbfcal{X}_U$ 
and $\mathbfcal{Y}_U$, i.e.,
	\begin{align*}
	\boldsymbol{x}_n\overset{n\rightarrow\infty}{\rightharpoonup}\boldsymbol{x}
	\text{ in }{\mathbfcal{X}_U}\cap_{\boldsymbol{j}_U}{\mathbfcal{Y}_U},
	\quad\mathbfcal{A}_U\boldsymbol{x}_n\overset{n\rightarrow\infty}{\rightharpoonup}
	\boldsymbol{\xi}_U\text{ in }\mathbfcal{X}_U^*,\quad
	\limsup_{n\rightarrow\infty}{\langle \mathbfcal{A}_U\boldsymbol{x}_n,
		\boldsymbol{x}_n\rangle_{\mathbfcal{X}_U}}\leq \langle\boldsymbol{\xi}_U,
	\boldsymbol{x}\rangle_{\mathbfcal{X}_U}.
	\end{align*}
    This and the embedding ${\mathbfcal{X}_U}\cap_{\boldsymbol{j}_U}{\mathbfcal{Y}_U}
    \hookrightarrow \mathbfcal{X}\cap_{\boldsymbol{j}}\mathbfcal{Y}$ 
    immediately imply $\boldsymbol{x}_n\overset{n\rightarrow\infty}{\rightharpoonup}
    \boldsymbol{x}\text{ in }\mathbfcal{X}\cap_{\boldsymbol{j}}\mathbfcal{Y}$.
	In addition, there exist constants $M,M'>0$ such that 
	$\|\boldsymbol{x}_n\|_{{\mathbfcal{X}}\cap_{\boldsymbol{j}}{\mathbfcal{Y}}}
	=\|\boldsymbol{x}_n\|_{{\mathbfcal{X}_U}\cap_{\boldsymbol{j}_U}{\mathbfcal{Y}_U}}\leq M$ 
	and $\|\mathbfcal{A}_U\boldsymbol{x}_n\|_{\mathbfcal{X}_U^*}\leq M'$
	for all $n\in\mathbb{N}$.
	As $\mathbfcal{B}:{\mathbfcal{X}}\cap_{\boldsymbol{j}}{\mathbfcal{Y}}
	\rightarrow \mathbfcal{X}^*$ is bounded, we obtain a further constant 
	$C>0$ such that $\|\mathbfcal{B}\boldsymbol{x}_n\|_{\mathbfcal{X}^*}
	\leq C$ for all $n\in\mathbb{N}$. From this and \eqref{eq:4.4} we 
	deduce for all $n\in \mathbb{N}$ that
	\begin{align*}
	\langle \mathbfcal{A}_0\boldsymbol{x}_n,\boldsymbol{x}_n
	\rangle_{\mathbfcal{X}}\leq \|\mathbfcal{A}_U\boldsymbol{x}_n\|_{\mathbfcal{X}_U^*}
	\|\boldsymbol{x}_n\|_{\mathbfcal{X}_U}+
	\|\mathbfcal{B}\boldsymbol{x}_n\|_{\mathbfcal{X}^*}
	\|\boldsymbol{x}_n\|_{\mathbfcal{X}}\leq 
	(M'+C)M.
	\end{align*}
	Since $\mathbfcal{A}_0:\mathbfcal{X}\rightarrow \mathbfcal{X}^*$ 
	is monotone Proposition \ref{2.2} with 
	$S=(\boldsymbol{x}_n)_{n\in\mathbb{N}}\subseteq \mathbfcal{X}$ 
	and $h\equiv1$ provides a constant $K>0$ such that 
	$\|\mathbfcal{A}_0\boldsymbol{x}_n\|_{\mathbfcal{X}^*}\leq K$ 
	for all $n\in\mathbb{N}$. Thus, 
	$(\mathbfcal{A}\boldsymbol{x}_n)_{n\in\mathbb{N}}\subseteq \mathbfcal{X}^*$ 
	is bounded and by dint of the reflexivity of $\mathbfcal{X}^*$ 
	we extract a subsequence 
	$(\mathbfcal{A}\boldsymbol{x}_n)_{n\in\Lambda}\subseteq \mathbfcal{X}^*$, 
	with $\Lambda\subseteq\mathbb{N}$, and an element
	$\boldsymbol{\xi}\in \mathbfcal{X}^*$ such that
	\begin{align*}
	\mathbfcal{A}\boldsymbol{x}_n
	\overset{n\rightarrow\infty}{\rightharpoonup}
	\boldsymbol{\xi}\quad\text{ in }\mathbfcal{X}^*\quad(n\in\Lambda).
	\end{align*}
	We infer $\mathbfcal{A}_U\boldsymbol{x}_n=(\text{id}_{\mathbfcal{X}_U})^*
	\mathbfcal{A}\boldsymbol{x}_n
	\overset{n\rightarrow\infty}{\rightharpoonup}(\text{id}_{\mathbfcal{X}_U})^*
	\boldsymbol{\xi}$ in $\mathbfcal{X}_U^*$ $(n\in\Lambda)$,
	 i.e. $(\text{id}_{\mathbfcal{X}_U})^*\boldsymbol{\xi}=
	 \boldsymbol{\xi}_U\text{ in }\mathbfcal{X}_U^*$,
	  from the weak continuity of $(\text{id}_{\mathbfcal{X}_U})^*
	  :\mathbfcal{X}^*\rightarrow\mathbfcal{X}_U^*$. Finally, 
	  we use \eqref{eq:4.4} once more to obtain
	\begin{align*}
	\limsup_{\substack{n\rightarrow\infty\\n\in\Lambda}}{\langle 
		\mathbfcal{A}\boldsymbol{x}_n,\boldsymbol{x}_n
		\rangle_{\mathbfcal{X}}}&=
	\limsup_{\substack{n\rightarrow\infty\\n\in\Lambda}}{\langle 
		\mathbfcal{A}_U\boldsymbol{x}_n,\boldsymbol{x}_n\rangle_{\mathbfcal{X}_U}}
	\leq \limsup_{n\rightarrow\infty}{\langle \mathbfcal{A}_U\boldsymbol{x}_n,
		\boldsymbol{x}_n\rangle_{\mathbfcal{X}_U}}\\&\leq \langle\boldsymbol{\xi}_U,
	\boldsymbol{x}\rangle_{\mathbfcal{X}_U}=\langle (\text{id}_{\mathbfcal{X}_U})^*
	\boldsymbol{\xi},\boldsymbol{x}\rangle_{\mathbfcal{X}_U}
	=\langle \boldsymbol{\xi},\boldsymbol{x}\rangle_{\mathbfcal{X}}.
	\end{align*}
	Altogether, $(\boldsymbol{x}_n)_{n\in\mathbb{N}}\subseteq 
	{\mathbfcal{X}}\cap_{\boldsymbol{j}}{\mathbfcal{Y}}$ satisfies 
	\eqref{eq:3.2}--\eqref{eq:3.5} with respect to $\mathbfcal{X}$ and 
	$\mathbfcal{Y}$, and the $C^0$-Bochner condition (M) of 
	$\mathbfcal{A}: \mathbfcal{X}\cap_{\boldsymbol{j}}\mathbfcal{Y}\rightarrow\mathbfcal{X}^*$ 
	finally yields $\mathbfcal{A}\boldsymbol{x}=\boldsymbol{\xi}$ in $\mathbfcal{X}^*$, 
	and therefore $\mathbfcal{A}_U\boldsymbol{x}=
	(\text{id}_{\mathbfcal{X}_U})^*\mathbfcal{A}\boldsymbol{x}
	=(\text{id}_{\mathbfcal{X}_U})^*\boldsymbol{\xi}=\boldsymbol{\xi}_U$ in 
	$\mathbfcal{X}_U^*$.\hfill$\square$
\end{proof}

\section{Main Theorem: (Separable case)}
\label{sec:5}

\begin{thm}\label{5.1}
	Let $(V,H,j)$ be an evolution triple, $V$ separable and $1<p<\infty$. 
	Furthermore, we require the following conditions:
	\begin{description}[(iii)]
		\item[(i)] $\mathbfcal{A}_0:\mathbfcal{X}\rightarrow \mathbfcal{X}^*$
		 is monotone.
		\item[(ii)] $\mathbfcal{B}:\mathbfcal{X}\cap_{\boldsymbol{j}}\mathbfcal{Y}
		\rightarrow \mathbfcal{X}^*$ is bounded.
		\item[(iii)] $\mathbfcal{A}:=\mathbfcal{A}_0+\mathbfcal{B}:
		\mathbfcal{X}\cap_{\boldsymbol{j}}\mathbfcal{Y}\rightarrow \mathbfcal{X}^*$
		 satisfies the $C^0$-Bochner condition (M) and is $C^0$-Bochner 
		 coercive with respect to $\boldsymbol{f}\in \mathbfcal{X}^*$ and $\boldsymbol{y}_0\in H$.
	\end{description}
	Then there exists a solution 
	$\boldsymbol{y}\in \mathbfcal{W}$ of 
	\begin{align*}
		\begin{alignedat}{2}
	\frac{d_e\boldsymbol{y}}{dt}+\mathbfcal{A}\boldsymbol{y}&=\boldsymbol{f}&&\quad\text{ in }\mathbfcal{X}^*,\\
	(\boldsymbol{j\boldsymbol{y}})(0)&=\boldsymbol{y}_0&&\quad\text{ in }H.
		\end{alignedat}
	\end{align*}
\end{thm}

\paragraph{\textbf{Proof}}

\paragraph{\textbf{0. Reduction of assumptions:}}
It suffices to treat the special case $\boldsymbol{f}=\boldsymbol{0}$
 in $\mathbfcal{X}^*$. Otherwise, we switch to 
 $\widehat{\mathbfcal{A}}:=\mathbfcal{A}_0+\widehat{\mathbfcal{B}}:
 \mathbfcal{X}\cap_{\boldsymbol{j}}\mathbfcal{Y}\rightarrow \mathbfcal{X}^*$ 
 with the shifted bounded part $\widehat{\mathbfcal{B}}:=
 \mathbfcal{B}-\boldsymbol{f}:\mathbfcal{X}\cap_{\boldsymbol{j}}
 \mathbfcal{Y}\rightarrow \mathbfcal{X}^*$. It is straightforward 
 to check that $\widehat{\mathbfcal{A}}$ still satisfies the 
 $C^0$-Bochner condition (M) and is $C^0$-Bochner coercive 
 with respect to $\boldsymbol{0}\in \mathbfcal{X}^*$ and $\boldsymbol{y}_0\in H$. 

\paragraph{\textbf{1. Galerkin approximation:}}
We apply the Galerkin approach of Section \ref{sec:4}. 
As Galerkin basis of $(V,\boldsymbol{y}_0)$ will serve a sequence 
$(V_n,ja_n)_{n\in\mathbb{N}}$ with the following properties:
\begin{itemize}
	\item[$\bullet$] $V_n \subseteq V_{n+1}\subseteq V,\;\dim V_n<\infty
	\text{ and }\overline{\bigcup_{n\in\mathbb{N}}{V_n}}^{\|.\|_V}=V$.
	\item[$\bullet$] $a_n\in V_n,\;ja_n\overset{n\rightarrow\infty}{\rightarrow}
	\boldsymbol{y}_0\text{ in }H\text{ and }\|ja_n\|_H\leq\|\boldsymbol{y}_0\|_H$.
\end{itemize}
The existence of such a sequence is a consequence of the 
separability of $V$ in conjunction with the given evolution triple structure. 
The well-posedness of the Galerkin system with respect to 
$(V_n,ja_n)_{n\in\mathbb{N}}$ follows as in Remark \ref{4.1}. 
We denote by $\boldsymbol{y}_n\in \mathbfcal{W}_{V_n}$ 
the \textbf{n.th Galerkin solution} if
\begin{align}\begin{split}
\begin{alignedat}{2}
\frac{d_{e_{V_n}}\boldsymbol{y}_n}{dt}+
\mathbfcal{A}_{V_n}\boldsymbol{y}_n&=\boldsymbol{0}&&\quad\text{ in }\mathbfcal{X}_{V_n}^*,\\
(\boldsymbol{j}_{V_n}\boldsymbol{y}_n)(0)&=ja_n&&\quad\text{ in }H_{V_n}.
\end{alignedat}
\end{split}\label{eq:5.2}
\end{align}

\paragraph{\textbf{2. Existence of Galerkin solutions:}} 
As the operator $\mathbfcal{A}$ is not necessarily induced, Carath\'eodory's 
theorem is not available.
However, we will prove the existence of Galerkin solutions similarly 
to Carath\'eodory's theorem by translating \eqref{eq:5.2} into an equivalent 
fixed point problem and then exploiting an appropriate version of 
Schauder's fixed point theorem. To this end, we first translate 
\eqref{eq:5.2} into an equivalent differential equation with values in $V_n$ 
instead of $V_n^*$, to have a chance to meet the in Schauder's fixed 
point theorem demanded self map property.

\paragraph{\textbf{2.1 Equivalent differential equation:}}
\hypertarget{2.1}{}
As $e_{V_n}:V_n\rightarrow V_n^*$ is an isomorphism 
Proposition \ref{2.4} ensures that  the induced operator 
$\boldsymbol{e}_n:L^{p'}(I,V_n)\rightarrow \mathbfcal{X}_{V_n}^*$, 
given via $(\boldsymbol{e}_n\boldsymbol{x})(t):=e_{V_n}(\boldsymbol{x}(t))$ 
for almost every $t\in I$ and all $\boldsymbol{x}\in L^{p'}(I,V_n)$, 
is also an isomorphism. Apart from this, Proposition~\ref{2.8} additionally 
yields that $\boldsymbol{e}_n:W^{1,p'}(I,V_n)\rightarrow W^{1,p'}(I,V_n^*)$ 
is an isomorphism and that for all $\boldsymbol{x}\in W^{1,p'}(I,V_n)$ it holds
\begin{align}
\frac{d_{V_n^*}\boldsymbol{e}_n\boldsymbol{x}}{dt}=
\boldsymbol{e}_n\frac{d_{V_n}\boldsymbol{x}}{dt}
\quad\text{ in  }\mathbfcal{X}_{V_n}^*.\label{eq:f}
\end{align} 
Using Proposition \ref{2.7} and \eqref{eq:2.7a} we see that 
$\boldsymbol{y}_n\in\mathbfcal{W}_{V_n}$ satisfies $\eqref{eq:5.2}_1$ 
if and only if $\boldsymbol{y}_n\in \mathbfcal{X}_{V_n}$ 
and $\boldsymbol{e}_n\boldsymbol{y}_n\in W^{1,p'}(I,V_n^*)$ with
\begin{align}
\frac{d_{V_n^*}\boldsymbol{e}_n\boldsymbol{y}_n}{dt}
=\frac{d_{e_{V_n}}\boldsymbol{y}_n}{dt}=
-\mathbfcal{A}_{V_n}\boldsymbol{y}_n\quad\text{ in }\mathbfcal{X}_{V_n}^*.\label{eq:zw}
\end{align}
By exploiting Proposition \ref{2.8} and \eqref{eq:f} we 
further deduce the equivalence of \eqref{eq:zw} and
$\boldsymbol{y}_n=\boldsymbol{e}_n^{-1}\boldsymbol{e}_n\boldsymbol{y}_n
\in W^{1,p'}(I,V_n)$ with
\begin{align*}
\frac{d_{V_n}\boldsymbol{y}_n}{dt}=
\boldsymbol{e}_n^{-1}\frac{d_{V_n^*}\boldsymbol{e}_n\boldsymbol{y}_n}{dt}=
-\boldsymbol{e}_n^{-1}\mathbfcal{A}_{V_n}\boldsymbol{y}_n\quad\text{ in }L^{p'}(I,V_n).
\end{align*}
Proposition \ref{2.9} (i) provides the choice function 
$\{\cdot\}_X:W^{1,p'}(I,V_n)\rightarrow C^0(\overline{I},V_n)$. Thus, 
$\boldsymbol{y}_n\in \mathbfcal{W}_{V_n}$ satisfies $\eqref{eq:5.2}_2$ 
if and only if $\{\boldsymbol{y}_n\}_{V_n}(0)=a_n$ in $V_n$ in the sense of 
the unique continuous representation $\{\boldsymbol{y}_n\}_{V_n}\in C^0(\overline{I},V_n)$. 
Altogether, $\boldsymbol{y}_n\in \mathbfcal{W}_{V_n}$ is a solution 
of \eqref{eq:5.2} if and only if $\boldsymbol{y}_n\in W^{1,p'}(I,V_n)$ with
\begin{align}
\begin{split}
\begin{alignedat}{2}
\frac{d_{V_n}\boldsymbol{y}_n}{dt}&=
-\boldsymbol{e}_n^{-1}\mathbfcal{A}_{V_n}\boldsymbol{y}_n&&\quad\text{ in }L^{p'}(I,V_n),\\
\{\boldsymbol{y}_n\}_{V_n}(0)&=a_n&&\quad\text{ in }V_n.
\end{alignedat}
\end{split}\label{eq:5.3}
\end{align}

\paragraph{\textbf{2.2 Equivalent fixed point problem:}}
\hypertarget{2.2}{}
From $C^0(\overline{I},V_n)=\mathbfcal{X}_{V_n}\cap_{\boldsymbol{j}_{V_n}}
\mathbfcal{Y}_{V_n}$ with norm equivalence and Proposition \ref{2.9} (ii) we 
deduce the well-definedness of the fixed point operator 
$\mathbfcal{F}_n:C^0(\overline{I},V_n)\rightarrow W^{1,p'}(I,V_n)$ defined by
\begin{align*}
(\mathbfcal{F}_n\boldsymbol{x})(t)&:=a_n-(\mathbfcal{V}\boldsymbol{e}_n^{-1}
\mathbfcal{A}_{V_n}\boldsymbol{x})(t)\\&\,=
a_n-\int_0^t{(\boldsymbol{e}_n^{-1}\mathbfcal{A}_{V_n}\boldsymbol{x})(s)\;ds}
\quad\text{ in }V_n\quad\text{ for all }t\in\overline{I}
\end{align*}
and every $\boldsymbol{x}\in C^0(\overline{I},V_n)$. In addition, the embedding $W^{1,p'}(I,V_n)\hookrightarrow C^0(\overline{I},V_n)$ 
(cf.~Proposition \ref{2.9} (i)) provides the well-definedness of 
$\mathbfcal{F}_n:W^{1,p'}(I,V_n)\subseteq C^0(\overline{I},V_n)\rightarrow W^{1,p'}(I,V_n)$. 
Analogously to the theory of ordinary differential equations we 
conclude under the renewed application of Proposition \ref{2.9} 
the equivalence of \eqref{eq:5.3} and the existence of a fixed point 
of $\mathbfcal{F}_n:W^{1,p'}(I,V_n)\subseteq C^0(\overline{I},V_n)\rightarrow W^{1,p'}(I,V_n)$.

\paragraph{\textbf{2.3 Existence of a fixed point of \textnormal{$\mathbfcal{F}_n$}:}}
The verification of the existence of a fixed point is based on the following 
version of Schauder's fixed point theorem:
\textit{
	\paragraph{\textit{Theorem 5.2: (Schauder, 1930)}}
	Let ${\cal(X,\|\cdot\|_X)}$ be a Banach space, ${\cal F:K\subseteq X\rightarrow K}$ 
	a continuous operator and ${\cal K\subseteq X}$ a non-empty, 
	convex and compact set. Then there exists $x\;{\cal\in K}$ such that
	\begin{align*}
	{\cal F}x=x\quad{\cal\text{ in }X.}
	\end{align*}
	\paragraph{\textit{Proof}} See \cite[Kapitel 1, Satz 2.46]{Ru04}.\hfill $\square$}\\
\\It remains to verify the assumptions of Schauder's fixed point theorem.
\begin{description}
	\item[(i)]\textbf{Continuity of }$\mathbfcal{F}_n$: 
	Lemma \ref{4.5} in conjunction with Proposition \ref{3.6} (i) 
	yields the demi-continuity of $\mathbfcal{A}_{V_n}:
	\mathbfcal{X}_{V_n}\cap_{\boldsymbol{j}_{V_n}}\mathbfcal{Y}_{V_n}
	\rightarrow \mathbfcal{X}_{V_n}^*$. Thus, as it holds 
	$C^0(\overline{I},V_n)=\mathbfcal{X}_{V_n}\cap_{\boldsymbol{j}_{V_n}}\mathbfcal{Y}_{V_n}$
	 with norm equivalence, and both $\boldsymbol{e}_n^{-1}:\mathbfcal{X}_{V_n}^*
	 \rightarrow L^{p'}(I,V_n)$ and $\mathbfcal{V}:L^{p'}(I,V_n)\rightarrow W^{1,p'}(I,V_n)$ 
	 are weakly continuous (cf.~Proposition \ref{2.4} and \ref{2.9} (ii)), 
	 $\mathbfcal{F}_n:C^0(\overline{I},V_n)\rightarrow W^{1,p'}(I,V_n)$ 
	 is demi-continuous. Proposition \ref{2.10} eventually provides the strong 
	 continuity of the embedding $W^{1,p'}(I,V_n)\hookrightarrow C^0(\overline{I},V_n)$ 
	 and consequently the continuity of 
	 $\mathbfcal{F}_n:C^0(\overline{I},V_n)\rightarrow C^0(\overline{I},V_n)$.
	\item[(ii)] \textbf{Self-map property of the compressed fixed point operator:} 
	Since $\mathbfcal{F}_n$  in general fails to comply with the in Schauder's 
	fixed point theorem demanded self-map property, we construct a compression 
	operator $\tau_n:C^0(\overline{I},V_n)\rightarrow\left(0,1\right]$ such that the 
	compressed operator $\tau_n\mathbfcal{F}_n$ meets the self-map property 
	and has coinciding fixed point set with $\mathbfcal{F}_n$. Then it suffices to 
	prove the existence of a fixed point of the compressed fixed point operator.\\
	\noindent\hspace*{5mm}As we are not aware of how to construct the desired 
	compression operator $\tau_n$ we first consider $\tau_n\mathbfcal{F}_n$, with an 
	arbitrary operator $\tau_n:C^0(\overline{I},V_n)\rightarrow\left(0,1\right]$, and 
	demonstrate the existence of a-priori estimates which are independent of this 
	operator.
	\begin{itemize}
		\item[(a)] \textbf{Invariance of the a-priori estimates with respect to compressions:}
		\hypertarget{2.3 (ii) (a)}{}
		We fix an arbitrary operator $\tau_n:C^0(\overline{I},V_n)\rightarrow\left(0,1\right]$ 
		and assume there exists a fixed point $\boldsymbol{y}_n\in W^{1,p'}(I,V_n)$ 
		of $\tau_n\mathbfcal{F}_n:W^{1,p'}(I,V_n)\subseteq C^0(\overline{I},V_n)
		\rightarrow W^{1,p'}(I,V_n)$. Then we deduce analogously to the discussion 
		in Step \hyperlink{2.1}{2.1} and \hyperlink{2.2}{2.2} that 
		$\boldsymbol{y}_n\in \mathbfcal{W}_{V_n}$ satisfies
		\begin{align}
		\begin{split}
		\begin{alignedat}{2}
		\frac{d_{e_{V_n}}\boldsymbol{y}_n}{dt}+\tau_n(\boldsymbol{y}_n)
		\mathbfcal{A}_{V_n}\boldsymbol{y}_n&=\boldsymbol{0}&&\quad\text{ in }\mathbfcal{X}_{V_n}^*,\\
		(\boldsymbol{j}_{V_n}\boldsymbol{y}_n)(0)&
		=\tau_n(\boldsymbol{y}_n) ja_n&&\quad\text{ in }H_{V_n}.
		\end{alignedat}
		\end{split}\label{eq:5.4}
		\end{align} 
		Testing \eqref{eq:5.4} by $\boldsymbol{y}_n\chi_{\left[0,t\right]}
		\in \mathbfcal{X}_{V_n}$, 
		where $t\in\left(0,T\right]$ is arbitrary, and a subsequent application 
		of the generalized integration by parts formula \eqref{eq:4.3} and identity \eqref{eq:4.4} with $U=V_n$ yield
		\begin{align}
		\begin{split}
		\tau_n(\boldsymbol{y}_n)\langle \mathbfcal{A}\boldsymbol{y}_n,
		\boldsymbol{y}_n\chi_{\left[0,t\right]}\rangle_\mathbfcal{X}&
		=\tau_n(\boldsymbol{y}_n)\langle\mathbfcal{A}_{V_n}\boldsymbol{y}_n,\boldsymbol{y}_n
		\chi_{\left[0,t\right]}\rangle_{\mathbfcal{X}_{V_n}}\\&=-\left\langle\frac{d_{e_{V_n}}
			\boldsymbol{y}_n}{dt},\boldsymbol{y}_n\chi_{\left[0,t\right]}\right
		\rangle_{\mathbfcal{X}_{V_n}}\\&=-\frac{1}{2}\|(\boldsymbol{j}_{V_n}
		\boldsymbol{y}_n)(t)\|_H^2+\frac{\tau_n(\boldsymbol{y}_n)^2}{2}\|ja_n\|_H^2.
		\end{split}\label{eq:5.5}
		\end{align}
		From dividing \eqref{eq:5.5} by $0<\tau_n(\boldsymbol{y}_n)\leq 1$ and 
		using $\|ja_n\|_H\leq \|\boldsymbol{y}_0\|_H$ we further obtain 
		\begin{align}
		\frac{1}{2\tau_n(\boldsymbol{y}_n)}\|(\boldsymbol{j}\boldsymbol{y}_n)(t)\|_H^2
		+\langle \mathbfcal{A}\boldsymbol{y}_n,\boldsymbol{y}_n\chi_{\left[0,t\right]}
		\rangle_{\mathbfcal{X}}\leq\frac{\tau_n(\boldsymbol{y}_n)}{2}\| \boldsymbol{y}_0\|_H^2
		\leq\frac{1}{2}\| \boldsymbol{y}_0\|_H^2.\label{eq:5.6}
		\end{align}
		As $t\in \overline{I}$ was arbitrary, $2^{-1}\leq (2\tau_n(\boldsymbol{y}_n))^{-1}$ 
		and $\mathbfcal{A}:\mathbfcal{X}\cap_{\boldsymbol{j}}\mathbfcal{Y}\rightarrow 
		\mathbfcal{X}^*$ is $C^0$-Bochner coercive with respect to 
		$\boldsymbol{0}\in\mathbfcal{X}^*$ and $\boldsymbol{y}_0\in H$ there exists an $n$-independent 
		constant $M>0$ such that
		\begin{align}
		\|\boldsymbol{y}_n\|_{\mathbfcal{X}\cap_{\boldsymbol{j}}\mathbfcal{Y}}
		\leq M.\label{eq:5.7}
		\end{align} 
		The boundedness of $\mathbfcal{B}:\mathbfcal{X}\cap_{\boldsymbol{j}}
		\mathbfcal{Y}\rightarrow \mathbfcal{X}^*$ and \eqref{eq:5.7} further yield 
		an $n$-independent constant $C>0$ such that 
		$\|\mathbfcal{B}\boldsymbol{y}_n\|_{\mathbfcal{X}^*} \leq C$. 
		From this, \eqref{eq:5.6} in the case $t=T$, and \eqref{eq:5.7} we obtain
		\begin{align*}
		\langle \mathbfcal{A}_0\boldsymbol{y}_n,\boldsymbol{y}_n\rangle_\mathbfcal{X}&
		=\langle \mathbfcal{A}\boldsymbol{y}_n,\boldsymbol{y}_n\rangle_\mathbfcal{X}
		-\langle \mathbfcal{B}\boldsymbol{y}_n,\boldsymbol{y}_n\rangle_\mathbfcal{X}
		\\&\leq \frac{1}{2}\| \boldsymbol{y}_0\|_H^2+CM.
		\end{align*}
		Finally, Proposition \ref{2.2} with $S=(\boldsymbol{y}_n)_{n\in\mathbb{N}}
		\subseteq\mathbfcal{X}$ and $h\equiv1$ provides an $n$-independent 
		constant $M'>0$ such that
		\begin{align}
		\|\mathbfcal{A}\boldsymbol{y}_n\|_{\mathbfcal{X}^*}\leq M'.\label{eq:5.8}
		\end{align}
		\item[(b)] \textbf{Construction of the compression operator:} The demi-continuity of 
		$\mathbfcal{A}:\mathbfcal{X}\cap_{\boldsymbol{j}}\mathbfcal{Y}\rightarrow\mathbfcal{X}^*$ 
		(cf.~Propositions \ref{2.2} (i) and \ref{3.6} (ii)) and the embedding 
		$C^0(\overline{I},V_n)\hookrightarrow\mathbfcal{X}\cap_{\boldsymbol{j}}\mathbfcal{Y}$ 
		imply the continuity of
		\begin{align*}
		(\boldsymbol{x}\mapsto \vert\langle \mathbfcal{A}\boldsymbol{x},\boldsymbol{x}\rangle_{\mathbfcal{X}}\vert):
		C^0(\overline{I},V_n)\rightarrow\mathbb{R}_{\ge 0}\qquad
		\text{ and }\qquad\|\cdot\|_{\mathbfcal{X}\cap_{\boldsymbol{j}}\mathbfcal{Y}}:
		C^0(\overline{I},V_n)\rightarrow \mathbb{R}_{\ge 0}.
		\end{align*}
		From this we deduce the continuity of 
		$g,h:C^0(\overline{I},V_n)\rightarrow \mathbb{R}_{\ge 0}$ 
		defined by
		\begin{alignat}{2}
		g(\boldsymbol{x})&:=
		\begin{cases}
		1&\text{, if }\|\boldsymbol{x}\|_{\mathbfcal{X}\cap_{\boldsymbol{j}}\mathbfcal{Y}}\leq 2M\\
		\frac{2M}{\|\boldsymbol{x}\|_{\mathbfcal{X}\cap_{\boldsymbol{j}}\mathbfcal{Y}}}&\text{, else}
		\end{cases},\label{eq:5.9}
		\\h(\boldsymbol{x})&:=
		\begin{cases}
		1&\text{, if }\vert\langle \mathbfcal{A}\boldsymbol{\boldsymbol{x}},\boldsymbol{x}\rangle_{\mathbfcal{X}}\vert\leq M'M\\
		\frac{M'M}{\vert\langle \mathbfcal{A}\boldsymbol{x},\boldsymbol{x}\rangle_{\mathbfcal{X}}\vert}&\text{, else}
		\end{cases},\label{eq:5.10}
		\end{alignat}
		for every $\boldsymbol{x}\in C^0(\overline{I},V_n)$. Finally, the compression operator
		\begin{align}
		\tau_n:=(\boldsymbol{x}\mapsto g(h(\boldsymbol{x})\mathbfcal{F}_n\boldsymbol{x})h(\boldsymbol{x})):C^0(\overline{I},V_n)
		\rightarrow \left(0,1\right],\label{eq:5.11}
		\end{align}
		and therefore the compressed fixed point operator 
		$\tau_n \mathbfcal{F}_n:C^0(\overline{I},V_n)\rightarrow C^0(\overline{I},V_n)$ 
		are continuous.
		\item[(c)] \textbf{Equivalence of the fixed point problems:}
		\hypertarget{(2.3) (ii) (c)}{}
		Since $\tau_n:C^0(\overline{I},V_n)\rightarrow \left(0,1\right]$ was still an 
		arbitrary operator in \eqref{eq:5.4}, the a-priori estimates \eqref{eq:5.7} 
		and \eqref{eq:5.8} hold true for both $\tau_n\equiv 1$ and the compression 
		operator defined in \eqref{eq:5.11}. Being more precise, a fixed point 
		$\boldsymbol{y}_n\in W^{1,p'}(I,V_n)$ of $\tau_n\mathbfcal{F}_n$ as 
		well as a fixed point of $\mathbfcal{F}_n$ satisfies the estimates
		\begin{align*}
		\|\boldsymbol{y}_n\|_{\mathbfcal{X}\cap_{\boldsymbol{j}}\mathbfcal{Y}}
		\leq M \qquad\text{ and }\qquad\|\mathbfcal{A}\boldsymbol{y}_n\|_{\mathbfcal{X}^*}\leq M',
		\end{align*}
		with the same $n$-independent constants $M,M'>0$. These imply 
		$\vert\langle \mathbfcal{A}\boldsymbol{y}_n,\boldsymbol{y}_n\rangle_\mathbfcal{X}\vert\leq M'M$ 
		and therefore $h(\boldsymbol{y}_n)=1$ due to the definition of 
		$h$ (cf.~\eqref{eq:5.10}). From this we are able to derive the 
		equivalence of the fixed point problems. In fact, there holds:
		
		\noindent\hspace*{5mm}If $\boldsymbol{y}_n\in W^{1,p'}(I,V_n)$ is a 
		fixed point of $\mathbfcal{F}_n$, then having regard to 
		$\|\boldsymbol{y}_n\|_{\mathbfcal{X}\cap_{\boldsymbol{j}}\mathbfcal{Y}}\leq M<2M$ 
		and the definition of $g$ (cf.~\eqref{eq:5.9}) we deduce
		\begin{align*}
		g(h(\boldsymbol{y}_n)\mathbfcal{F}_n\boldsymbol{y}_n)
		=g(\mathbfcal{F}_n\boldsymbol{y}_n)=g(\boldsymbol{y}_n)=1.
		\end{align*}
		From this we obtain $\tau_n(\boldsymbol{y}_n)=1$ and thus 
		$\tau_n(\boldsymbol{y}_n)\mathbfcal{F}_n\boldsymbol{y}_n=
		\mathbfcal{F}_n\boldsymbol{y}_n=\boldsymbol{y}_n$ in $W^{1,p'}(I,V_n)$. 
		
		\noindent\hspace*{5mm}On the other hand, if $\boldsymbol{y}_n\in W^{1,p'}(I,V_n)$ 
		is a fixed point of the compressed operator $\tau_n\mathbfcal{F}_n$, then
        \begin{align*}
            \tau_n(\boldsymbol{y}_n)
            =g(h(\boldsymbol{y}_n)\mathbfcal{F}_n\boldsymbol{y}_n)h(\boldsymbol{y}_n)
            =1
        \end{align*}
        has to be valid.
		Otherwise, taking into account $h(\boldsymbol{y}_n)=1$ and 
		the definition of $g$ (cf.~\eqref{eq:5.9}),
		$\|\mathbfcal{F}_n\boldsymbol{y}_n\|_{\mathbfcal{X}\cap_{\boldsymbol{j}}\mathbfcal{Y}}> 2M$ would hold true. 
		But this yields the contradiction
		\begin{align*}
		M&\ge \|\boldsymbol{y}_n\|_{\mathbfcal{X}\cap_{\boldsymbol{j}}\mathbfcal{Y}}
		=\|\tau_n(\boldsymbol{y}_n)\mathbfcal{F}_n\boldsymbol{y}_n\|_{\mathbfcal{X}\cap_{\boldsymbol{j}}\mathbfcal{Y}}
		=\|g\left(h(\boldsymbol{y}_n)\mathbfcal{F}_n\boldsymbol{y}_n\right)h(\boldsymbol{y}_n)
		\mathbfcal{F}_n\boldsymbol{y}_n\|_{\mathbfcal{X}\cap_{\boldsymbol{j}}\mathbfcal{Y}}\\&
		=\|g\left(\mathbfcal{F}_n\boldsymbol{y}_n\right)\mathbfcal{F}_n
		\boldsymbol{y}_n\|_{\mathbfcal{X}\cap_{\boldsymbol{j}}\mathbfcal{Y}}
		=\frac{2M}{\|\mathbfcal{F}_n\boldsymbol{y}_n\|_{\mathbfcal{X}\cap_{\boldsymbol{j}}\mathbfcal{Y}}}
		\|\mathbfcal{F}_n\boldsymbol{y}_n\|_{\mathbfcal{X}\cap_{\boldsymbol{j}}\mathbfcal{Y}}=2M>M.
		\end{align*}
		As a consequence, it holds $\mathbfcal{F}_n\boldsymbol{y}_n
		=\tau_n(\boldsymbol{y}_n)\mathbfcal{F}_n\boldsymbol{y}_n=\boldsymbol{y}_n$ 
		in $W^{1,p'}(I,V_n)$.
		\item[(d)] \textbf{Existence of a fixed point of the equivalent compressed problem:}
		\hypertarget{2.3 (ii) (d)}{} We set 
		${\boldsymbol{\mu}_n:=\vertiii{\boldsymbol{e}_n^{-1}}_{\mathcal{L}(\mathbfcal{X}_{V_n}^*,L^{p'}(I,V_n))}}$. 
		Then for arbitrary $\boldsymbol{x}\in C^0(\overline{I},V_n)$ holds:
		\begin{fleqn}[0pt]
			\begin{align}
			&\|\tau_n(\boldsymbol{x})\mathbfcal{F}_n\boldsymbol{x}\|_{\mathbfcal{X}\cap_{\boldsymbol{j}}\mathbfcal{Y}}
			=\begin{cases}\|h(\boldsymbol{x})\mathbfcal{F}_n\boldsymbol{x}\|_{\mathbfcal{X}\cap_{\boldsymbol{j}}\mathbfcal{Y}}&\text{, if }\|h(\boldsymbol{x})\mathbfcal{F}_n\boldsymbol{x}\|_{\mathbfcal{X}\cap_{\boldsymbol{j}}\mathbfcal{Y}}\leq 2M\\
			\frac{2M\left\|h(\boldsymbol{x})\mathbfcal{F}_n\boldsymbol{x}\right
				\|_{\mathbfcal{X}\cap_{\boldsymbol{j}}\mathbfcal{Y}}}
			{\left\|h(\boldsymbol{x})\mathbfcal{F}_n\boldsymbol{x}\right\|_{\mathbfcal{X}\cap_{\boldsymbol{j}}\mathbfcal{Y}}}
			&\text{, else }
			\end{cases}\Biggr\rbrace\leq 2M.\label{eq:5.12}
			\end{align}
			\begin{align}
			\begin{split}
				\left\|\frac{d_{V_n}\tau_n(\boldsymbol{x})\mathbfcal{F}_n\boldsymbol{x}}{dt}\right\|_{L^{p'}(I,V_n)}&
				=\left\|\frac{d_{V_n}}{dt}(\tau_n(\boldsymbol{x})
				\mathbfcal{V}\boldsymbol{e}_n^{-1}\mathbfcal{A}_{V_n}\boldsymbol{x})\right\|_{L^{p'}(I,V_n)}\\
			&\leq\left\|h(\boldsymbol{x})\frac{d_{V_n}}{dt}(\mathbfcal{V}
			\boldsymbol{e}_n^{-1}\mathbfcal{A}_{V_n}\boldsymbol{x})\right\|_{L^{p'}(I,V_n)}
			\\&= \left\|h(\boldsymbol{x})\boldsymbol{e}_n^{-1}(\text{id}_{\mathbfcal{X}_{V_n}}
			)^*\mathbfcal{A}\boldsymbol{x}\right\|_{L^{p'}(I,V_n)}
			\\&\leq \boldsymbol{\mu}_n\|(\text{id}_{\mathbfcal{X}_{V_n}})^*h(\boldsymbol{x})
			\mathbfcal{A}\boldsymbol{x}\|_{\mathbfcal{X}_{V_n}^*}\leq  \boldsymbol{\mu}_n\|h(\boldsymbol{x})\mathbfcal{A}\boldsymbol{x}\|_{\mathbfcal{X}^*}.
			\end{split}\label{eq:5.13}
			\end{align}
		\end{fleqn}
		In the first inequality in \eqref{eq:5.13} we made use of 
		$\vert g(h\cdot\mathbfcal{F}_n)\vert\leq  1$. The subsequent equal sign 
		and inequality stem from $\frac{d_{V_n}}{dt}\mathbfcal{V}=\text{id}_{L^{p'}(I,V_n)}$ 
		(cf.~Proposition \ref{2.9} (ii)), the definition of $\mathbfcal{A}_{V_n}$ 
		(cf.~Remark \ref{4.1}) and $\vertiii{(\text{id}_{\mathbfcal{X}_{V_n}})^*}_{\mathcal{L}(\mathbfcal{X}^*,\mathbfcal{X}_{V_n}^*)}
		=\vertiii{\text{id}_{\mathbfcal{X}_{V_n}}}_{\mathcal{L}(\mathbfcal{X}_{V_n},\mathbfcal{X})}\leq 1$.
		
		\noindent\hspace*{5mm}Next, we fix the closed ball
		\begin{align*}
		\mathbfcal{S}:=B_{2M}^{\mathbfcal{X}\cap_{\boldsymbol{j}}\mathbfcal{Y}}(0)
		\subseteq \mathbfcal{X}\cap_{\boldsymbol{j}}\mathbfcal{Y}.
		\end{align*}
		From the boundedness of $\mathbfcal{B}:\mathbfcal{X}
		\cap_{\boldsymbol{j}}\mathbfcal{Y}\rightarrow\mathbfcal{X}^*$ we obtain a constant 
		$K_0>0$ such that $\|\mathbfcal{B}\boldsymbol{x}\|_{\mathbfcal{X}^*}\leq K_0$ for 
		all $\boldsymbol{x}\in \mathbfcal{S}$. Due to the definition of $h$ (cf.~\eqref{eq:5.10}) 
		we further deduce
		\begin{align*}
		h(\boldsymbol{x})\langle \mathbfcal{A}_0\boldsymbol{x},\boldsymbol{x}\rangle_\mathbfcal{X}&
		=h(\boldsymbol{x})\langle \mathbfcal{A}\boldsymbol{x},\boldsymbol{x}\rangle_\mathbfcal{X}-
		h(\boldsymbol{x})\langle \mathbfcal{B}\boldsymbol{x},\boldsymbol{x}\rangle_\mathbfcal{X}\\&
		\leq M'M+2K_0M
		\end{align*}
		for all $\boldsymbol{x}\in \mathbfcal{S}$.
		Proposition \ref{2.2} with $S=\mathbfcal{S}$ and $h$ as in \eqref{eq:5.10} provides a constant $K_1>0$ such 
		that $\|h(\boldsymbol{x})\mathbfcal{A}\boldsymbol{x}\|_{\mathbfcal{X}^*}\leq K_1$
		for all $\boldsymbol{x}\in  \mathbfcal{S}$. This and \eqref{eq:5.13} imply
		\begin{align}
		\left\|\frac{d_{V_n}\tau_n(\boldsymbol{x})\mathbfcal{F}_n\boldsymbol{x}}{dt}
		\right\|_{L^{p'}(I,V_n)}\leq \boldsymbol{\mu}_nK_1\label{eq:5.14}
		\end{align}
		for all $\boldsymbol{x}\in\mathbfcal{S}\cap C^0(\overline{I},V_n)$.
		Finally, we define
		\begin{align*}
		&\mathbfcal{B}_n:=\bigg\{\boldsymbol{x}\in W^{1,p'}(I,V_n)\cap  
		\mathbfcal{S}\;\bigg|\; \left\|\frac{d_{V_n}\boldsymbol{x}}{dt}\right\|_{L^{p'}(I,V_n)}
		\leq \boldsymbol{\mu}_nK_1\bigg\}\subseteq W^{1,p'}(I,V_n),
		\\&\mathbfcal{K}_n:=\overline{\mathbfcal{B}_n}^{\|\cdot\|_{C^0(\overline{I},V_n)}}
		\subseteq C^0(\overline{I},V_n).
		\end{align*}
		From \eqref{eq:5.12} and \eqref{eq:5.14} we derive the self map property 
		for $\tau_n\mathbfcal{F}_n$ on $\mathbfcal{K}_n$. In particular, 
		$\tau_n\mathbfcal{F}_n:\mathbfcal{K}_n\subseteq C^0(\overline{I},V_n)
		\rightarrow \mathbfcal{K}_n$ is well-defined and continuous. Thus, in view of Schauder's fixed point theorem it 
		remains to inquire into the properties of $\mathbfcal{K}_n$. To this end, 
		let us focus on $\mathbfcal{B}_n$, which is obviously non-empty, bounded 
		and convex in $W^{1,p'}(I,V_n)$. Proposition \ref{2.10} provides the compact 
		embedding
		\begin{align*}
		W^{1,p'}(I,V_n)\hookrightarrow\hookrightarrow C^0(\overline{I},V_n).
		\end{align*}
		In consequence, $\mathbfcal{K}_n$ is non-empty, convex and compact in 
		$C^0(\overline{I},V_n)$ and therefore Schauder's fixed point theorem applicable.
		According to Step \hyperlink{(2.3) (ii) (c)}{2.3 (ii) (c)} the existing fixed point is 
		also a fixed point of $\mathbfcal{F}_n$ and looking back to Step 
		\hyperlink{2.2}{2.2} and \hyperlink{2.1}{2.1} a solution of both \eqref{eq:5.3} 
		and \eqref{eq:5.2}.
	\end{itemize}
\end{description}	

\paragraph{\textbf{3. Passage to the limit:}}
\hypertarget{3}{}
\paragraph{\textbf{3.1 Convergence of Galerkin solutions:}}
\hypertarget{3.1}{}
Due to Step \hyperlink{2.3 (ii) (a)}{2.3 (ii) (a)} the verified solution 
$\boldsymbol{y}_n\in \mathbfcal{W}_{V_n}$ of \eqref{eq:5.2} satisfies the 
estimates \eqref{eq:5.7} and \eqref{eq:5.8}. In virtue of the reflexivity of 
$\mathbfcal{X}$ and $\mathbfcal{X}^*$ together with the existing separable 
pre-dual $L^1(I,H^*)$ of $L^\infty(I,H)\cong (L^1(I,H^*))^*$ (c.f.~\cite[Theorem 3.3]{BT38}) 
we obtain a not relabeled subsequence $(\boldsymbol{y}_n)_{n\in\mathbb{N}}\subseteq 
\mathbfcal{X}\cap_{\boldsymbol{j}}\mathbfcal{Y}$ as well as elements 
$\boldsymbol{y}\in \mathbfcal{X}\cap_{\boldsymbol{j}}L^\infty(I,H)$ and 
$\boldsymbol{\xi}\in \mathbfcal{X}^*$ such that
\begin{align}
\begin{alignedat}{2}
\boldsymbol{y}_n&\overset{n\rightarrow\infty}{\rightharpoonup}\boldsymbol{y}&\quad
&\text{ in }\mathbfcal{X},\\
\boldsymbol{j}\boldsymbol{y}_n&\;\;\overset{\ast}{\rightharpoondown}\;\;\boldsymbol{j}\boldsymbol{y}&&\text{
    in }L^\infty(I,H)\quad (n\rightarrow\infty),\\
\mathbfcal{A}\boldsymbol{y}_n&\overset{n\rightarrow\infty}{\rightharpoonup}\boldsymbol{\xi}&&\text{
    in }\mathbfcal{X}^*.
\end{alignedat}\label{eq:5.15}
\end{align}

\paragraph{\textbf{3.2 Regularity and trace of the weak limit:}}
\hypertarget{3.2}
Let $v\in V_k$, $k\in\mathbb{N}$, and $\varphi\in C^\infty(\overline{I})$ with $\varphi(T)=0$. 
Testing \eqref{eq:5.2} for all $n\ge k$ by $v\varphi\in\mathbfcal{X}_{V_k}\subseteq \mathbfcal{X}_{V_n}$ 
and a subsequent application of  the
generalized integration by parts formula \eqref{eq:4.3} with $U=V_n$ yield for all $n\ge k$
\begin{align}
\langle \mathbfcal{A}_{V_n}\boldsymbol{y}_n,v\varphi\rangle_{\mathbfcal{X}_{V_n}}
=-\left\langle \frac{d_{e_{V_n}}\boldsymbol{y}_n}{
	dt},v\varphi\right\rangle_{\mathbfcal{X}_{V_n}}
=\langle e_{V_n}(v)\varphi^\prime,\boldsymbol{y}_n\rangle_{\mathbfcal{X}_{V_n}}
+(ja_n,jv)_H\varphi(0).\label{eq:5.16}
\end{align}
Together with \eqref{eq:4.2} and \eqref{eq:4.4} in the case $U=V_n$, \eqref{eq:5.16} reads
\begin{align*}
\langle \mathbfcal{A}\boldsymbol{y}_n,v\varphi\rangle_\mathbfcal{X}
=\langle e(v)\varphi^\prime,\boldsymbol{y}_n\rangle_\mathbfcal{X}+(ja_n,jv)_H\varphi(0)
\end{align*}
for all $n\ge k$. By passing with $n\ge k$ to infinity, using \eqref{eq:5.15} 
and $ja_n\overset{n\rightarrow\infty}{\rightarrow}\boldsymbol{y}_0$ in $H$, 
we obtain
\begin{align}
\begin{split}
\langle \boldsymbol{\xi},v\varphi\rangle_\mathbfcal{X}
=\langle e(v)\varphi^\prime,\boldsymbol{y}\rangle_\mathbfcal{X}+(\boldsymbol{y}_0,jv)_H\varphi(0)
\end{split}\label{eq:5.17}
\end{align}
for all $v\in \bigcup_{k\in\mathbb{N}}{V_k}$ and $\varphi\in C^\infty(\overline{I})$ 
with $\varphi(T)=0$. In the case $\varphi\in C_0^\infty(I)$, \eqref{eq:5.17} reads
\begin{align*}
\langle e(v)\varphi^\prime,\boldsymbol{y}\rangle_\mathbfcal{X}
=\langle \boldsymbol{\xi},v\varphi\rangle_\mathbfcal{X}
\end{align*} 
for all $v\in \bigcup_{k\in\mathbb{N}}{V_k}$ and Proposition \ref{2.11} thus proves
\begin{align}
    \boldsymbol{y}\in \mathbfcal{W}\quad\text{ with }\quad\frac{d_e\boldsymbol{y}}{dt}
    =-\boldsymbol{\xi}\quad\text{ in }\mathbfcal{X}^*\qquad
    \text{ and }\qquad\boldsymbol{j}\boldsymbol{y}\in \mathbfcal{Y}.\label{eq:W}
\end{align}
In addition, we are allowed to apply the generalized integration by parts formula 
(cf.~Proposition \ref{2.12}) in \eqref{eq:5.17} in the case $\varphi\in C^\infty(\overline{I})$ 
with $\varphi(T)=0$ and $\varphi(0)=1$. In so doing, we further deduce that
\begin{align}
((\boldsymbol{j}\boldsymbol{y})(0)-\boldsymbol{y}_0,jv)_H=0\label{eq:5.18}
\end{align}
for all $v\in \bigcup_{k\in\mathbb{N}}{V_k}$. As $R(j)$ is dense in $H$ 
we obtain from \eqref{eq:5.18} that
\begin{align}
	(\boldsymbol{j}\boldsymbol{y})(0)=\boldsymbol{y}_0\quad\text{ in }H.\label{eq:5.19}
\end{align}

\paragraph{\textbf{3.3 Weak convergence in $\mathbfcal{X}\cap_{\boldsymbol{j}}\mathbfcal{Y}$:}}
The objective in the following passage is to exploit the characterization of 
weak convergence in $\mathbfcal{X}\cap_{\boldsymbol{j}}\mathbfcal{Y}$ 
(cf.~Proposition \ref{2.6}). To be more precise, it remains to verify pointwise 
weak convergence in $H$.

To this end, let us fix an arbitrary $t\in\left(0,T\right]$. From the a-priori estimate 
$\sup_{n\in\mathbb{N}}{\|\boldsymbol{j}\boldsymbol{y}_n\|_{\mathbfcal{Y}}}\leq M$ 
we obtain the existence of a subsequence 
$((\boldsymbol{j}\boldsymbol{y}_n)(t) )_{n\in\Lambda_t}\subseteq H$ 
with $\Lambda_t\subseteq\mathbb{N}$, initially depending on this 
fixed $t$, and an element $\boldsymbol{y}_{\Lambda_t}\in H$ such that
\begin{align}
(\boldsymbol{j}\boldsymbol{y}_n)(t) &\overset{n\rightarrow\infty}{\rightharpoonup}
\boldsymbol{y}_{\Lambda_t}\quad\text{ in }H\quad(n\in \Lambda_t).\label{eq:5.20}
\end{align}
For $v\in V_k$, $k\in\Lambda_t$, and
$\varphi\in C^\infty(\overline{I})$ with $\varphi(0)=0$ and
$\varphi(t)=1$, we test \eqref{eq:5.2} for $n\ge k$
($n\in \Lambda_t$) by
$v\varphi\chi_{\left[0,t\right]}\in \mathbfcal{X}_k\subseteq
\mathbfcal{X}_n$, use the
generalized integration by parts formula \eqref{eq:4.3}, \eqref{eq:4.2} and \eqref{eq:4.4} with $U=V_n$ to obtain for all
$n\ge k$ with $n\in \Lambda_t$ 
\begin{align*}
\langle \mathbfcal{A}\boldsymbol{y}_n,v\varphi\chi_{\left[0,t\right]}\rangle_\mathbfcal{X}
=\langle e(v)\varphi^\prime\chi_{\left[0,t\right]},\boldsymbol{y}_n\rangle_\mathbfcal{X}
-((\boldsymbol{j}\boldsymbol{y}_n)(t),jv)_H.
\end{align*}
By passing for $n\ge k$ with $n\in \Lambda_t$ to infinity, 
using \eqref{eq:5.15} and \eqref{eq:5.20}, we obtain
\begin{align*}
\langle \boldsymbol{\xi},v\varphi\chi_{\left[0,t\right]}\rangle_\mathbfcal{X}
=\langle e(v)\varphi^\prime\chi_{\left[0,t\right]},\boldsymbol{y}\rangle_\mathbfcal{X}
-(\boldsymbol{y}_{\Lambda_t},jv)_H
\end{align*}
for all $v\in \bigcup_{k\in\Lambda_t}{V_k}$. The generalized 
integration by parts formula (cf.~Proposition \ref{2.12}) and \eqref{eq:W} provide
\begin{align}
((\boldsymbol{j}\boldsymbol{y})(t)-\boldsymbol{y}_{\Lambda_t},jv)_H=0\label{eq:5.21}
\end{align}
for all $v\in\bigcup_{k\in\Lambda_t}{V_k}$. Thanks to $V_k\subseteq V_{k+1}$ 
for all $k\in\mathbb{N}$ there holds $\bigcup_{k\in\Lambda_t}{V_k}
=\bigcup_{k\in\mathbb{N}}{V_k}$. Thus, $j(\bigcup_{k\in\Lambda_t}{V_k})$ 
is dense in $H$ and we obtain from \eqref{eq:5.21} that 
$(\boldsymbol{j}\boldsymbol{y})(t)=\boldsymbol{y}_{\Lambda_t}$ in $H$, i.e.,
\begin{align}
(\boldsymbol{j}\boldsymbol{y}_n)(t)\overset{n\rightarrow\infty}{\rightharpoonup}
(\boldsymbol{j}\boldsymbol{y})(t)\quad\text{ in }H\quad(n\in\Lambda_t).\label{eq:5.22}
\end{align}
As this argumentation stays valid for each subsequence of 
$((\boldsymbol{j}\boldsymbol{y}_n)(t) )_{n\in\mathbb{N}}\subseteq H$, 
$(\boldsymbol{j}\boldsymbol{y})(t)\in H$ is weak accumulation point of each 
subsequence of $((\boldsymbol{j}\boldsymbol{y}_n)(t))_{n\in\mathbb{N}}\subseteq H$. 
The standard convergence principle (cf.~\cite[Kap. I, Lemma 5.4]{GGZ74}) finally yields 
$\Lambda_t=\mathbb{N}$ in \eqref{eq:5.22}. Since $t\in \left(0,T\right]$ 
was arbitrary in \eqref{eq:5.22} and 
$(\boldsymbol{y}_n)_{n\in\mathbb{N}}\subseteq \mathbfcal{X}\cap_{\boldsymbol{j}}\mathbfcal{Y}$ 
is bounded (cf.~\eqref{eq:5.7}) we conclude using Proposition \ref{2.6} that
\begin{align}
\boldsymbol{y}_n\overset{n\rightarrow\infty}{\rightharpoonup}
\boldsymbol{y}\quad\text{ in }\mathbfcal{X}\cap_{\boldsymbol{j}}\mathbfcal{Y}.\label{eq:5.23}
\end{align}

\paragraph{\textbf{3.4 Identification of \textnormal{$\mathbfcal{A}\boldsymbol{y}$ }and $\boldsymbol{\xi}$:}}
\hypertarget{3.4}{}
As $\tau_n(\boldsymbol{y}_n)=1$ 
estimate \eqref{eq:5.6} with $t=T$ reads
\begin{align}
\langle \mathbfcal{A}\boldsymbol{y}_n,\boldsymbol{y}_n\rangle_\mathbfcal{X}
\leq -\frac{1}{2}\|(\boldsymbol{j}\boldsymbol{y}_n)(T)\|_H^2+\frac{1}{2}\|\boldsymbol{y}_0\|_H^2\label{eq:5.24}
\end{align}
for all $n\in \mathbb{N}$. The limit superior with respect to $n\in\mathbb{N}$ 
on both sides in \eqref{eq:5.24}, \eqref{eq:5.19}, \eqref{eq:5.22} with 
$\Lambda_t=\mathbb{N}$ in the case $t=T$, the weak 
lower semi-continuity of $\|\cdot\|_H$, the generalized integration by parts 
formula (cf.~Proposition \ref{2.12}) and \eqref{eq:W} yield
\begin{alignat}{2}\begin{split}
\limsup_{n\rightarrow\infty}{
	\langle \mathbfcal{A}\boldsymbol{y}_n,\boldsymbol{y}_n\rangle_\mathbfcal{X}}
&\leq -\frac{1}{2}\|(\boldsymbol{j}\boldsymbol{y})(T)\|_H^2
+\frac{1}{2}\|(\boldsymbol{j}\boldsymbol{y})(0)\|_H^2\\&
=-\left\langle \frac{d_e\boldsymbol{y}}{dt},\boldsymbol{y}\right\rangle_\mathbfcal{X}
=\langle \boldsymbol{\xi},\boldsymbol{y}\rangle_\mathbfcal{X}.
\end{split}\label{eq:5.25}
\end{alignat}
In view of $\eqref{eq:5.15}_3$, \eqref{eq:5.23} and \eqref{eq:5.25}
we conclude from the $C^0$-Bochner condition (M) of 
$\mathbfcal{A}:\mathbfcal{X}\cap_{\boldsymbol{j}}\mathbfcal{Y}\rightarrow \mathbfcal{X}^*$ 
that $\mathbfcal{A}\boldsymbol{y}=\boldsymbol{\xi}\text{ in }\mathbfcal{X}^*$. 
All things considered, we proved
\begin{align*}
\begin{alignedat}{2}
\frac{d_e\boldsymbol{y}}{dt}+\mathbfcal{A}\boldsymbol{y}&=\boldsymbol{0}&&\quad\text{ in }\mathbfcal{X}^*,\\
(\boldsymbol{j}\boldsymbol{y})(0)&=\boldsymbol{y}_0&&\quad\text{ in }H.
\end{alignedat}
\end{align*} 
This completes the proof of Theorem \ref{5.1}.\hfill$\square$

\section{Main theorem: (Purely reflexive case)}
\label{sec:6}
This section is concerned with the extension of 
Theorem \ref{5.1} to the case of purely reflexive $V$. 

\begin{thm}\label{6.1}
  Theorem \ref{5.1} stays valid if we omit the separability of $V$.
\end{thm}

A lack of separability of $V$ results in a non-existence of an increasing sequence 
of finite dimensional subspaces which approximates $V$ up to density. 
We circumvent this problem by regarding a probably uncountable system 
of separable subspaces. But this system might be orderless, such that the 
increasing structure, which was indispensable for the proof of Theorem \ref{5.1}, 
has to be generated locally. The latter will be guaranteed by the subsequent lemma. 
Then, we perform the passage to limit as in Theorem \ref{5.1} Step \hyperlink{3}{3} 
first locally and assemble the extracted local information to the desired global assertion afterwards. 

\begin{lem}\label{6.2}
	Let $(V,H,j)$ be an evolution triple, $1<p<\infty$, $M>0$, $\boldsymbol{y}_0\in H$ and
	\begin{align*}
	\mathcal{U}_{\boldsymbol{y}_0}:=\left\{U\subseteq V\mid (U,\|\cdot\|_V)
	\text{ is a separable Banach space},\;\boldsymbol{y}_0\in H_U:=\overline{j(U)}^{\|\cdot\|_V}\right\}.
	\end{align*}
	Then it holds:
	\begin{description}[(ii)]
		\item[(i)] $(U,\boldsymbol{y}_0)_{U\in \mathcal{U}_{\boldsymbol{y}_0}}$ is 
		a Galerkin basis of $(V,\boldsymbol{y}_0)$ in the sense of Remark \ref{4.1} (i).
		\item[(ii)] Suppose for a mapping 
		$\boldsymbol{\Psi}:\mathcal{U}_{\boldsymbol{y}_0}\rightarrow 2^{B^{\mathbfcal{X}}_M(0)}\setminus \{\emptyset\}$ 
		with
		\begin{align*}
		\boldsymbol{\Psi}(U)\subseteq \mathbfcal{X}_U:=L^p(I,U)\qquad
		\text{ and }\qquad\textbf{L}_U:=\bigcup_{\substack{Z\in\mathcal{U}_{\boldsymbol{y}_0}
				\\ Z\supseteq U}}{\boldsymbol{\Psi}(Z)}\subseteq 2^{B^{\mathbfcal{X}}_M(0)}
			\qquad \text{ for all }U\in\mathcal{U}_{\boldsymbol{y}_0}
		\end{align*}
		there exists $\boldsymbol{y}\in \bigcap_{U\in\mathcal{U}_{\boldsymbol{y}_0}}
		{\overline{\textbf{L}_U}^{\tau(\mathbfcal{X},\mathbfcal{X}^*)}}$. Then for all 
		$U\in \mathcal{U}_{\boldsymbol{y}_0}$ there exist sequences $(U_n)_{n\in\mathbb{N}}
		\subseteq\mathcal{U}_{\boldsymbol{y}_0}$ and $(\boldsymbol{y}_n)_{n\in\mathbb{N}}
		\subseteq \mathbfcal{X}$, with $U\subseteq U_n\subseteq U_{n+1}$ and 
		$\boldsymbol{y}_n\in \boldsymbol{\Psi}(U_n)$ for all $n\in\mathbb{N}$, such 
		that
		\begin{align*}
			\boldsymbol{y}_n\overset{n\rightarrow\infty}{\rightharpoonup}\boldsymbol{y}\quad\text{ in }\mathbfcal{X}.
		\end{align*}
	\end{description}
\end{lem}

\begin{proof}
	Point (i) follows right from the definition in Remark \ref{4.1}. The verification 
	of (ii) is a straightforward modification of \cite[Proposition 11]{BH72}. For a detailed proof we refer to \cite[Lemma 7.1]{alex-master}.\hfill$\square$
\end{proof}

\paragraph{\textbf{Proof} (of Theorem \ref{6.1})} It suffices anew to treat 
the special case $\boldsymbol{f}=\boldsymbol{0}$ in $\mathbfcal{X}^*$.

\paragraph{\textbf{1. Galerkin approximation:}} Let $\mathcal{U}_{\boldsymbol{y}_0}$ 
be as in Lemma \ref{6.2}. In line with Remark \ref{4.1} we see the well-posedness of the 
Galerkin system with respect to $(U,\boldsymbol{y}_0)_{U\in\;\mathcal{U}_{\boldsymbol{y}_0}}$. 
We denote by $\boldsymbol{y}_U\in\mathbfcal{W}_U$ the \textbf{Galerkin solution with respect to }$U$ if
\begin{align}
\begin{split}
\begin{alignedat}{2}
\frac{d_{e_U}\boldsymbol{y}_U}{dt}+\mathbfcal{A}_U\boldsymbol{y}_U&
=\boldsymbol{0}&&\quad\text{ in }\mathbfcal{X}_U^*,\\
(\boldsymbol{j}_U\boldsymbol{y}_U)(0)&=\boldsymbol{y}_0&&\quad\text{ in }H_U.
\end{alignedat}
\end{split}\label{eq:6.3}
\end{align}

\paragraph{\textbf{2. Existence of Galerkin solutions:}}
Each $U\in\;\mathcal{U}_{\boldsymbol{y}_0}$ is separable and reflexive. 
Thus, it remains to inquire into the properties of the restricted operator 
$\mathbfcal{A}_U:=(\text{id}_{\mathbfcal{X}_U})^*\mathbfcal{A}:\mathbfcal{X}_U
\cap_{\boldsymbol{j}_U}\mathbfcal{Y}_U\rightarrow\mathbfcal{X}_U^*$. 
Lemma \ref{4.5} immediately provides:
\begin{description}
	\item[(i)] $(\mathbfcal{A}_0)_U:=(\text{id}_{\mathbfcal{X}_U})^*\mathbfcal{A}_0:
	\mathbfcal{X}_U\rightarrow\mathbfcal{X}_U^*$ is monotone.
	\item[(ii)] $\mathbfcal{B}_U:=(\text{id}_{\mathbfcal{X}_U})^*\mathbfcal{B}:
	\mathbfcal{X}_U\cap_{\boldsymbol{j}_U}\mathbfcal{Y}_U\rightarrow\mathbfcal{X}_U^*$ is bounded.
	\item[(iii)] $\mathbfcal{A}_U:=(\mathbfcal{A}_0)_U+\mathbfcal{B}_U:
	\mathbfcal{X}_U\cap_{\boldsymbol{j}_U}\mathbfcal{Y}_U\rightarrow\mathbfcal{X}_U^*$ 
	satisfies the $C^0$-Bochner condition (M) and is $C^0$-Bochner coercive 
	with respect to $\boldsymbol{0}\in \mathbfcal{X}_U^*$ and $\boldsymbol{y}_0\in H_U$.
\end{description}
All things considered, Theorem \ref{5.1} yields the solvability of \eqref{eq:6.3} 
for all $U\in\;\mathcal{U}_{\boldsymbol{y}_0}$. In addition, we obtain as in Theorem 
\ref{5.1} Step \hyperlink{2.3 (ii) (a)}{2.3 (ii) (a)} $U$-independent constants $M,M'>0$ 
such that
\begin{align}
\|\boldsymbol{y}_U\|_{\mathbfcal{X}\cap_{\boldsymbol{j}}\mathbfcal{Y}}\leq M
\qquad\text{ and }\qquad\|\mathbfcal{A}\boldsymbol{y}_U\|_{\mathbfcal{X}^*}\leq M'.\label{eq:6.4}
\end{align}
Therefore, the mapping $\boldsymbol{\Psi}:\mathcal{U}_{\boldsymbol{y}_0}
\rightarrow 2^{B^{\mathbfcal{X}}_M(0)}\setminus\{\emptyset\}$, given via
\begin{align}
\boldsymbol{\Psi}(U):=\{\boldsymbol{y}_U\in \mathbfcal{W}_U\mid 
\boldsymbol{y}_U\text{ solves \eqref{eq:6.3} with respect to }U\}
\subseteq\mathbfcal{X}_U,\qquad U\in\mathcal{U}_{\boldsymbol{y}_0},\label{eq:6.5}
\end{align}
is well-defined. Apart from this, we define $\textbf{L}_U:=
\bigcup_{\substack{Z\in\mathcal{U}_{\boldsymbol{y}_0}\\ Z\supseteq U}}
{\boldsymbol{\Psi}(Z)}\neq\emptyset$.

\paragraph{\textbf{3. Passage to the limit:}}
Our next objective is to show
\begin{align}
	\bigcap_{U\in\mathcal{U}_{\boldsymbol{y}_0}}
	{\overline{\textbf{L}_U}^{\tau(\mathbfcal{X},\mathbfcal{X}^*)}}\neq\emptyset.\label{eq:6.6}
\end{align}
Then Lemma \ref{6.2} is applicable and we are in the position to perform the passage 
to the limit as in Theorem \ref{5.1} Step \hyperlink{3}{3} locally for each 
$U\in\mathcal{U}_{\boldsymbol{y}_0}$.

By construction holds $\textbf{L}_Z\subseteq \textbf{L}_W$ for all $Z,W\in\mathcal{U}_{\boldsymbol{y}_0}$ 
with $W\subseteq Z$. As $\overline{\langle Z\cup W\rangle}^{\|\cdot\|_V}\in \;\mathcal{U}_{\boldsymbol{y}_0}$ 
for all $Z,W\in\mathcal{U}_{\boldsymbol{y}_0}$, we thus have
\begin{align}
\emptyset\neq \textbf{L}_{\overline{\langle Z\cup W\rangle}^{\|\cdot\|_V}}\subseteq \textbf{L}_Z\cap \textbf{L}_W\label{eq:6.7}
\end{align}
for all $Z,W\in\mathcal{U}_{\boldsymbol{y}_0}$. By induction we obtain from \eqref{eq:6.7} 
that $(\overline{\textbf{L}_U}^{\tau(\mathbfcal{X},\mathbfcal{X}^*)})_{U\in\;\mathcal{U}_{\boldsymbol{y}_0}}$ 
satisfies the finite intersection property. As $B_M^{\mathbfcal{X}}(0)$ is compact with respect to 
$\tau(\mathbfcal{X},\mathbfcal{X}^*)$ and $\textbf{L}_U\subseteq B_M^{\mathbfcal{X}}(0)$ 
(cf.~\eqref{eq:6.4}), we conclude \eqref{eq:6.6} from the finite intersection principle 
(cf.~\cite[Appendix, Lemma 1.3]{Ru04}).\\

Now we perform the passage to the limit locally for each $U\in\mathcal{U}_{\boldsymbol{y}_0}$. 
To this end, we fix an arbitrary $U\in\mathcal{U}_{\boldsymbol{y}_0}$. Then Lemma \ref{6.2} 
provides sequences $(U_n)_{n\in\mathbb{N}}\subseteq \mathcal{U}_{\boldsymbol{y}_0}$ and 
$(\boldsymbol{y}_n)_{n\in\mathbb{N}}\subseteq\mathbfcal{X}\cap_{\boldsymbol{j}}\mathbfcal{Y}$, 
with $U\subseteq U_n\subseteq U_{n+1}$ and $\boldsymbol{y}_n\in \boldsymbol{\Psi}(U_n)\subseteq \mathbfcal{X}_{U_n}$ 
for all $n\in\mathbb{N}$, such that for all $n\in \mathbb{N}$
\begin{align}
\boldsymbol{y}_n\overset{n\rightarrow\infty}{\rightharpoonup}&
\;\boldsymbol{y}\quad\:\,\,\text{ in }\mathbfcal{X},\label{eq:6.8}\\
\begin{split}
\frac{d_{e_{U_n}}\boldsymbol{y}_n}{dt}+\mathbfcal{A}_{U_n}\boldsymbol{y}_n
=&\;\boldsymbol{0}\quad\;\,\,\text{ in }\mathbfcal{X}_{U_n}^*,\\
(\boldsymbol{j}_{U_n}\boldsymbol{y}_n)(0)=&\;\boldsymbol{y}_0\,\quad\text{ in }H_{U_n}.
\end{split}\label{eq:6.9}
\end{align}
For this specific $U\in\mathcal{U}_{\boldsymbol{y}_0}$ and $(U_n)_{n\in\mathbb{N}}$ 
from Lemma \ref{6.2} we set 
$U_\infty:=\overline{\text{span}\{\bigcup_{n\in\mathbb{N}}{U_n}\}}^{\|\cdot\|_V}\in \mathcal{U}_{\boldsymbol{y}_0}$. 
Then $U_\infty$ is a separable, reflexive Banach space and $\mathbfcal{A}_{U_\infty}:\mathbfcal{X}_{U_\infty}\cap_{\boldsymbol{j}_{U_\infty}}
\mathbfcal{Y}_{U_\infty}\rightarrow\mathbfcal{X}_{U_\infty}^*$ satisfies 
the assumptions of Theorem \ref{5.1} with respect to $U_\infty$ in the role of $V$. 
As $\mathbfcal{X}_{U_\infty}$ is closed with respect to $\tau(\mathbfcal{X},\mathbfcal{X}^*)$ 
and $(\boldsymbol{y}_n)_{n\in\mathbb{N}}\subseteq \mathbfcal{X}_{U_\infty}
\cap_{\boldsymbol{j}_{U_\infty}}\mathbfcal{Y}_{U_\infty}$, we deduce from 
$\eqref{eq:6.8}$ that $\boldsymbol{y}\in \mathbfcal{X}_{U_\infty}$ and
\begin{align}
\boldsymbol{y}_n\overset{n\rightarrow\infty}{\rightharpoonup}
\;\boldsymbol{y}\quad\text{ in }\mathbfcal{X}_{U_\infty}.\label{eq:6.10}
\end{align}
From $\|\boldsymbol{j}_{U_\infty}\boldsymbol{y}_n\|_{\mathbfcal{Y}_{U_\infty}}
=\|\boldsymbol{j}\boldsymbol{y}_n\|_{\mathbfcal{Y}}\leq M$ and 
$\|\mathbfcal{A}_{U_\infty}\boldsymbol{y}_n\|_{\mathbfcal{X}_{U_\infty}^*}
\leq\|\mathbfcal{A}\boldsymbol{y}_n\|_{\mathbfcal{X}^*}\leq M'$
we additionally obtain a subsequence $(\boldsymbol{y}_n)_{n\in\Lambda_{U_\infty}}
\subseteq \mathbfcal{X}_{U_\infty}\cap_{\boldsymbol{j}_{U_\infty}}\mathbfcal{Y}_{U_\infty}$, 
with $\Lambda_{U_\infty}\subseteq\mathbb{N}$, and 
$\boldsymbol{\xi}_{U_\infty}\in \mathbfcal{X}_{U_\infty}^*$ such that
\begin{align}
\begin{alignedat}{2}
\boldsymbol{j}_{U_\infty}\boldsymbol{y}_n&\;\;\overset{\ast}{\rightharpoondown}
\;\;\boldsymbol{j}_{U_\infty}\boldsymbol{y}&\quad
&\text{ in }L^\infty(I,H_{U_\infty})\quad(n\rightarrow\infty,n\in\Lambda_{U_\infty}),\\
\mathbfcal{A}_{U_\infty}\boldsymbol{y}_n&\overset{n\rightarrow\infty}{\rightharpoonup}
\boldsymbol{\xi}_{U_\infty}&&\text{ in }\mathbfcal{X}^*_{U_\infty}\quad(n\in \Lambda_{U_\infty}).
\end{alignedat}\label{eq:6.11}
\end{align}
All things considered, we are now in the situation of Theorem \ref{5.1} Step
 \hyperlink{3.1}{3.1} if 
$U_\infty$ takes the role of $V$ with Galerkin basis $(U_n,\boldsymbol{y}_0)_{n\in\mathbb{N}}$. 
Thus, we recapitulate Step \hyperlink{3.1}{3.1} till \hyperlink{3.4}{3.4} in Theorem \ref{5.1}. 
In doing so, we need to replace \eqref{eq:5.2} by \eqref{eq:6.8} and \eqref{eq:5.12} by 
\eqref{eq:6.9} together with \eqref{eq:6.10}. Thus, we infer that $\boldsymbol{y}\in \mathbfcal{W}_{U_\infty}$ with
\begin{align}
\begin{split}
\begin{alignedat}{2}
\frac{d_{e_{U_\infty}}\boldsymbol{y}}{dt}+\mathbfcal{A}_{U_\infty}\boldsymbol{y}
&=\boldsymbol{0}&&\quad\text{ in }\mathbfcal{X}_{U_\infty}^*,\\
(\boldsymbol{j}_{U_\infty}\boldsymbol{y})(0)&=\boldsymbol{y}_0&&\quad\text{ in }H_{U_\infty}.
\end{alignedat}
\end{split}\label{eq:6.12}
\end{align}
Let $u\in U\subseteq U_{\infty}$ and $\varphi\in C_0^\infty(I)$. Testing \eqref{eq:6.12} by 
$u\varphi\in \mathbfcal{X}_{U_\infty}$ and a subsequent application of the generalized integration by parts formula \eqref{eq:4.3}, \eqref{eq:4.2} and
\eqref{eq:4.4} with $U=U_\infty$ provide
\begin{align}
\langle \mathbfcal{A}\boldsymbol{y},u\varphi\rangle_{\mathbfcal{X}}
=\langle \mathbfcal{A}_{U_\infty}\boldsymbol{y},u\varphi\rangle_{\mathbfcal{X}_{U_\infty}}
=\langle e_{U_\infty}(u)\varphi^\prime,\boldsymbol{y}\rangle_{\mathbfcal{X}_{U_\infty}}
=\langle e(u)\varphi^\prime,\boldsymbol{y}\rangle_{\mathbfcal{X}}\label{eq:6.13}
\end{align}
for all $u\in U$ and $\varphi\in C_0^\infty(I)$. As $U\in\mathcal{U}_{\boldsymbol{y}_0}$ was arbitrary, 
\eqref{eq:6.13} is actually valid for all $u\in V$. Therefore, Proposition \ref{2.11} 
proves $\boldsymbol{y}\in\mathbfcal{W}$ with
\begin{align*}
\begin{alignedat}{2}
\frac{d_e\boldsymbol{y}}{dt}+\mathbfcal{A}\boldsymbol{y}
&=\boldsymbol{0}&&\quad\text{ in }\mathbfcal{X}^*,\\
(\boldsymbol{j}\boldsymbol{y})(0)&=\boldsymbol{y}_0&&\quad\text{ in }H.
\end{alignedat}
\end{align*}
This completes the proof of Theorem \ref{6.1}.\hfill $\square$
\newline

With the help of Theorem \ref{6.1} we are able to extend the results in \cite{KR19} to the case of purely reflexive $V$.

\begin{cor}\label{6.14}
	Let $(V,H,j)$ be an evolution triple, $1<p<\infty$ and $A(t):V\to V^*$, $t\in I$, a family of operators satisfying (\hyperlink{C.1}{C.1})--(\hyperlink{C.4}{C.4}). Then for arbitrary $\boldsymbol{y}_0\in H$ and $\boldsymbol{f}\in \mathbfcal{X}^*$ there exists a solution $\boldsymbol{y}\in\mathbfcal{W}$ of \eqref{eq:1}.
\end{cor}

\begin{proof}
	Immediate consequence of Theorem \ref{6.1} and Lemma \ref{HL}, since Bochner pseudo-monotonicity and Bochner coercivity imply $C^0$-Bochner pseudo-monotonicity and $C^0$-Bochner coercivity.\hfill$\square$
\end{proof}

There is still some room for improvement. Indeed, the statements of Theorem \ref{6.1} and Corollary \ref{6.14} remain true under more general assumptions. For proofs we refer to \cite{alex-master}.

\begin{rmk}\label{6.15}
	\textbf{(i)} Corollary \ref{6.14} remains true if we replace the evolution triple $(V,H,j)$ by a pre-evolution triple $(V,H,j)$, i.e., not $V$ but $V\cap_j H$ embeds continuously and dense into $H$ (cf.~\cite{KR19} or \cite[Definition 8.1]{alex-master}).
	
	\textbf{(ii)} Theorem \ref{6.1} remains true if we replace $\mathbfcal{X}=L^p(I,V)$ by the intersection $\mathbfcal{X}=L^p(I,V)\cap_{\boldsymbol{j}} L^q(I,H)$, where $1<p\leq q<\infty$ and $(V,H,j)$ is a pre-evolution triple (cf.~\cite[Satz 8.7]{alex-master}).
\end{rmk}

\appendix
\section{Pull-back intersections}
\label{sec:7}
This passage is highly inspired by \cite[Chapter 3]{BS88}. 
For proofs we refer to \cite{KR19}.

\begin{defn}[Embedding]\label{7.1}
	Let $(X,\tau_X)$ and $(Y,\tau_Y)$ be topological vector spaces. 
	The operator $j:X\rightarrow Y$ is said to be an \textbf{embedding} 
	if it is linear, injective and continuous. In this case we use the notation
	\begin{align*}
	X\overset{j}{\hookrightarrow}Y.
	\end{align*}
	If $X\subseteq Y$ and $j=\text{id}_X$, then we write $X\hookrightarrow Y$ instead. 
\end{defn}

\begin{defn}[Compatible couple]\label{7.2}
	Let $(X,\|\cdot\|_X)$ and $(Y,\|\cdot\|_Y)$ be Banach spaces such that embeddings 
	$e_X:X\rightarrow Z$ and $e_Y:Y\rightarrow Z$ into a Hausdorff vector space 
	$(Z,\tau_Z)$ exist. Then we call $(X,Y):=(X,Y,Z,e_X,e_Y)$ a \textbf{compatible couple}.
\end{defn}

\begin{defn}[Pull-back intersection of Banach spaces]\label{7.3}
	Let $(X,Y)$ be a compatible couple. Then the operator 
	$j:=e_Y^{-1}e_X:e_X^{-1}(R(e_X)\cap R(e_Y))\rightarrow Y$
	is well-defined and we denote by
	\begin{align*}
	X\cap_j Y:=e_X^{-1}(R(e_X)\cap R(e_Y))\subseteq X
	\end{align*}
	the \textbf{pull-back intersection of $X$ and $Y$ in $X$ with respect to $j$}. 
	Furthermore, $j$ is said to be the \textbf{corresponding intersection embedding}. 
	If $X,Y\subseteq Z$ with $e_X=\text{id}_X$ and $e_Y=\text{id}_Y$, 
	then we set $X\cap Y:=X\cap_jY$.
\end{defn}

\begin{prop}[Completeness of $X\cap_j Y$]\label{7.4}
	Let $(X,Y)$ be a compatible couple. Then $X\cap_j Y$ is a vector space 
	and equipped with norm
	\begin{align*}
	\|\cdot\|_{X\cap_j Y}:=\|\cdot\|_X+\|j\cdot\|_Y
	\end{align*}
	a Banach-space. Moreover, $j:X\cap_j Y\to Y$ is an embedding.
\end{prop}

\begin{prop}[Properties of $X\cap_j Y$]\label{7.5}
	Let $(X,Y)$ be a compatible couple. Then it holds:
	\begin{description}[(iii)]
		\item[(i)] If $X$ and $Y$ are reflexive or separable, then $X\cap_j Y$ is as well.
		\item[(ii)] \textbf{First characterization of weak convergence in $X\cap_j Y$:} A sequence $(x_n)_{n\in \mathbb{N}}\subseteq X\cap_j Y$ and $x\in X\cap_j Y$ satisfy
		$x_n\overset{n\rightarrow\infty}{\rightharpoonup}x\text{ in }X\cap_j Y$ 
		if and only if
		$x_n\overset{n\rightarrow\infty}{\rightharpoonup}x\text{ in }X$ and 
		$jx_n\overset{n\rightarrow\infty}{\rightharpoonup}jx\text{ in }Y$.
        \item[(iii)] \textbf{Second characterization of weak convergence in $X\cap_j Y$:} 
        In addition, let $X$ be reflexive. A sequence $(x_n)_{n\in \mathbb{N}}\subseteq X\cap_j Y$ and $x\in X\cap_j Y$ satisfy
        $x_n\overset{n\rightarrow\infty}{\rightharpoonup}x\text{ in }X\cap_j Y$ 
        if and only if $\sup_{n\in\mathbb{N}}{\|x_n\|_X}<\infty$ and 
        $jx_n\overset{n\rightarrow\infty}{\rightharpoonup}jx\text{ in }Y$.
	\end{description}
\end{prop}

\subsection*{Acknowledgments}
I would like to thank Michael R\r{u}\v{z}i\v{c}ka for all his helpful advices during the preparation of this paper.

\ifx\undefined\bysame
\newcommand{\bysame}{\leavevmode\hbox to3em{\hrulefill}\,}
\fi

\end{document}